%% file: main.tex
\documentclass[reqno, 12pt]{amsart} 

\usepackage{amssymb,bbm, bbold}
\usepackage{amsfonts}
\usepackage{amsmath, amsthm, amscd}
\usepackage{graphicx}
\usepackage{color}
\usepackage{enumerate}

\usepackage{hyperref}

\usepackage{tikz}
\usetikzlibrary{automata}
\usetikzlibrary{arrows,shapes,calc}
\usetikzlibrary{positioning}

\makeatletter
\def\l@subsection{\@tocline{2}{0pt}{2.5pc}{5pc}{}}
\renewcommand\tocchapter[3]{%
  \indentlabel{\@ifnotempty{#2}{\ignorespaces#2.\quad}}#3%
}
\newcommand\@dotsep{4.5}
\def\@tocline#1#2#3#4#5#6#7{\relax
  \ifnum #1>\c@tocdepth 
  \else
    \par \addpenalty\@secpenalty\addvspace{#2}%
    \begingroup \hyphenpenalty\@M
    \@ifempty{#4}{%
      \@tempdima\csname r@tocindent\number#1\endcsname\relax
    }{%
      \@tempdima#4\relax
    }%
    \parindent\z@ \leftskip#3\relax \advance\leftskip\@tempdima\relax
    \rightskip\@pnumwidth plus1em \parfillskip-\@pnumwidth
    #5\leavevmode\hskip-\@tempdima{#6}\nobreak
    \leaders\hbox{$\m@th\mkern \@dotsep mu\hbox{.}\mkern \@dotsep mu$}\hfill
    \nobreak
    \hbox to\@pnumwidth{\@tocpagenum{#7}}\par
    \nobreak
    \endgroup
  \fi}
\makeatother
\AtBeginDocument{%
\makeatletter
\expandafter\renewcommand\csname r@tocindent0\endcsname{0pt}
\makeatother
}
\def\l@subsection{\@tocline{2}{0pt}{2.5pc}{5pc}{}}

\newtheorem{theorem}{Theorem}[section]
\newtheorem{proposition}[theorem]{Proposition}
\newtheorem{lemma}[theorem]{Lemma}
\newtheorem{corollary}[theorem]{Corollary}

\newtheorem{remark}{Remark}[section]
\newtheorem{example}{Example}[section]
\newtheorem{claim}{Claim}[section]

\newcommand{\s}{s}    

\newcommand{\comp}{\circ}

\newcommand{\e}{\varepsilon}
\newcommand{\N}{\mathbb{N}}
\newcommand{\Z}{\mathbb{Z}}

\newcommand{\Pd}{\mathbb{P}}
\newcommand{\Ed}{\mathbb{E}}

\newcommand{\dist}[3]{\mathrm{dist}_{T_{#1}}(#2,#3)}
\newcommand{\degr}[2]{\mathrm{deg}_{T_{#1}}(#2)}

\newcommand{\Name}{\rm{TBRW}}
\newcommand{\Root}{{\rm{root}}}

\newcommand{\bxi}{\boldsymbol{\xi}}

\newenvironment{claimproof}[1]{\par\noindent\textit{Proof of the claim:}\space#1}{\hfill $\blacksquare$}

\makeatletter
\@addtoreset{equation}{section}
\makeatother

\renewcommand\thetable{\thesection.\@arabic\c@table}

\bibdata{prob}
\bibstyle{alpha}

\title[Tree Builder Random Walk]{Tree Builder Random Walk: recurrence, transience and ballisticity}

\author{Giulio Iacobelli$^1$, Rodrigo Ribeiro$^2$, Glauco Valle$^3$ and Leonel Zuazn\'abar$^4$}
\thanks{2. Supported by The Stochastic Models of Disordered and Complex Systems (NC120062) supported by the Millenium Scientific Initiative of the Ministry of Science and Technology  (Chile). }
\thanks{3. Supported by CNPq grant 305805/2015-0, Universal CNPq project 421383/2016-0 and FAPERJ grant E-26/203.048/2016.}
\thanks{4. Supported by PNPD/CAPES grant  88882.315481/2013-01}

\address{
\newline
\newline
$^1$ UFRJ - Instituto de Matem\'atica.
\newline  Caixa Postal 68530, 21945-970, Rio de Janeiro, Brasil.
\newline
e-mail: {\rm \texttt{giulio@im.ufrj.br}}
\newline
\newline
$^2$ PUC Chile.
\newline  Av. Vicu\~{n}a Mackenna 4860, Macul, La Florida, Regi\'{o}n Metropolitana, Chile.
\newline
e-mail: {\rm \texttt{rribeiro@impa.br}}
\newline
\newline
$^3$ UFRJ - Instituto de Matem\'atica.
\newline  Caixa Postal 68530, 21945-970, Rio de Janeiro, Brasil.
\newline
e-mail: {\rm \texttt{glauco.valle@im.ufrj.br}}
\newline
\newline
$^4$ 
USP - Instituto de Matem\'atica e Estat\'istica.
\newline  R. do Mat\~ao, 1010 - Butant\~a, S\~ao Paulo - SP, CEP: 05508-090, Brasil.
\newline
e-mail:{\rm \texttt{lzuaznabar@ime.usp.br}}
}



\begin{document}

\maketitle

\begin{abstract}
The Tree Builder Random Walk is a special random walk that evolves on trees whose size increases with time, randomly and depending upon the walker. After every $s$ steps of the walker, a random number of vertices are added to the tree and attached to the current position of the walker. These processes share similarities with other important classes of markovian and non-markovian random walks presenting a large variety of behaviors according to parameters specifications. We show that for a large and most significant class of tree builder random walks, the process is either null recurrent or transient. If $s$ is odd, the walker is ballistic and thus transient. If $s$ is even, the walker's behavior  can be explained from local properties of the growing tree and it can be either null recurrent or it gets trapped on some limited part of the growing tree. 
\end{abstract}

\tableofcontents

\input{introduction}

\input{model}

\input{loop-process}

\input{recurrence-s-even-Rodrigo}

\input{ballisticity-s-odd-Rodrigo}

\input{height}

\input{finalcomments}




%

\end{document}

%% file: introduction.tex

\section{Introduction}
\label{sec:intro}

In this paper we consider a special random walk, which we call Tree Builder Random Walk (\Name). It evolves on trees whose size increases randomly with time. Specifically, given $s\in \mathbb{N}$ (a parameter of the model), after every $s$ transitions of the walker a random number of vertices are added to the tree and attached to the current position of the random walk. As will become clear from our results, the \Name{} has an intrinsic mathematical interest connected to other important classes of markovian and non-markovian random walks such as Random Walks in Random Environment~\cite{SZ, OZ}, Reinforced Random Walks~\cite{D, DST, RP1988, RP2007} and Excited Random Walks~\cite{KZ}. It is also important to mention that the study of random graphs~\cite{H, JLR}, and random graphs dynamics~\cite{Du2} has been a very active field of research motivated by the increasing applicability of network models to represent real life phenomena.   Many  interesting questions are related to the evolution of random walks on 
these growing/random networks~\cite{A2018, DHS, LPP, PSS}. 
Despite the similarities with these models, the \Name{} possesses the distinctive feature of having the evolution of the graph dependent upon the walker's position,  see \cite{FIN} and references therein. 
%

In order to present adequately our results and draw connections to previous works let us first introduce formally the model.

%% file: model.tex

\subsection{The model}
\label{sec:model}

Let $T$ be a tree and  denote by $V(T)$ and $E(T)$ its vertex and edge sets, respectively. Let $\Omega$ be the collection of pairs $(T,x)$, where $T$ is a tree and $x \in V(T)$ is one of its  vertices. Now fix a locally finite tree $T_0$, a positive integer $s$ and a sequence of non-negative integer random variables $\boldsymbol{\xi} = \{\xi_{n}\}_{n\in \N}$.
The \Name{}  is a stochastic processes $\{(T_n ,X_n)\}_{n \ge 0}$ on $\Omega$ ($T_n$ denotes the tree  at time $n$ and $X_n$  one of its vertices)  defined inductively on almost every realization of $\boldsymbol{\xi}$ according to the update rules below.
\begin{enumerate}
	\item Obtain a locally finite tree $T_{n+1}$ from $T_{n}$ as follows:
	
	if $n= 0\,{\rm mod}\, s$, add  $\xi_{n}$  new leaves to $X_{n}$,
	
	if $n\neq 0\,{\rm mod}\, s$,  $T_{n+1}=T_{n}$.
	\item Choose uniformly one edge in $\{\{X_n, y\}: \{X_n, y\} \in E(T_{n+1}) \}$, i.e., an edge incident to $X_n$ in $T_{n+1}$,  and set  $X_{n+1}$ as the chosen neighbor of $X_n$. 
\end{enumerate}
We stress out  the subscript $n+1$ of $T$ in~(2); it means that we may add a new neighbor at (1) and  choose it at (2). If $\boldsymbol{\xi}$ is a sequence of independent random variables, the \Name{} process is a Markov chain.

Note that $s$ and the sequence of random variables $\boldsymbol{\xi}$ are parameters of the model. The first one allows the tree to grow only at times multiple of $s$, whereas the second controls the growth of the tree; for this reason we call the sequence $\boldsymbol{\xi}$ \textit{environment process}.  We denote by $\Pd_{T_0,x_0, s, \xi}(\cdot)$ the law of~$\{( T_n ,X_n)\}_{n \in \mathbb{N}}$ when~$( T_0 ,X_0)=(T_0, x_0)$, and by $\mathbb{E}_{T_0,x_0, s, \xi}(\cdot)$ the corresponding expectation.

For the sake of simplicity we will consider some nomenclature that will be useful. Given a realization of $T_n$, consider $x\in V(T_n)$, if $x$ has a single neighbor in $T_n$, we say that $x$ is a {\it leaf} of $T_n$ or a {\it leaf} of $z \in V(T_n)$ if $z$ is the single neighbor of $x$ in $T_n$.
We also remark that we can allow~$T_0$ to be a single vertex with a self-loop. This will be explained in the following sections.

\bigskip 

The \Name{} model generalizes a couple of models which have been recently studied; the NRRW (No Restart Random Walk) \cite{FIN} and the BGRW (Bernoulli Growth Random Walk)~\cite{FIORR}. In particular, \Name{} reduces to NRRW assuming $\xi_n \equiv 1\; \forall n$, whereas it reduces to BGRW  assuming $\s=1$ and $\xi_n\sim {\rm Ber}(p)\; \forall n$. 

%
%
%
%

\subsection{Environment conditions}\label{sec:environment}

The main goal of this paper is to provide conditions on the environment process $\boldsymbol{\xi} = \{\xi_{j}\}_{j \in \N}$ under which we observe recurrence or transience of~$\{X_n\}_{n \in \N}$. Since the process $\boldsymbol{\xi}$ controls how the walk modifies its environment, we refer to any distributional condition on $\boldsymbol{\xi}$ as \textit{environment condition}. In this section we list all conditions that will be used throughout the paper together with some brief discussion and examples. We reserve the letters $P$ and $E$ for the marginal distribution  of $\boldsymbol{\xi}$ and the corresponding expectation.

Two basic hypothesis on the variables $\xi_n$, $n\ge 1$, are that they are independent or even i.i.d. In the first case we say that $\boldsymbol{\xi}$ is an independent environment and in the second that it is an i.i.d. environment. For the other ones we reserve special notation. The first condition, denoted by (UE) is the following one
\begin{equation}\tag{UE}\label{cond:UE}
\inf_{n\in \N} P\left( \xi_{n}\geq 1\right)=\kappa>0\;.
\end{equation}
Tracing a parallel with the classical theory of random walk on random environment, the above condition is similar in spirit with the uniformly elliptic condition also denoted by~\eqref{cond:UE}. In our case, whenever the walk can add a new leaf to its environment, it has bounded away from zero probability of adding at least one leaf. This fact will be crucial to prove ballisticity when  $s$ is odd, since we can use this property to ``force routes of escape" as explained in Section~\ref{sec:odd}.

The next condition imposes restrictions on the moments of the environment process. Given $r > 0$, we say that $\boldsymbol{\xi}$ satisfies condition (\ref{cond:Mr})$_r$ if
\begin{equation}\tag{M}\label{cond:Mr}
\sup_{n\in \N} E(\xi_{n}^r) \le M < \infty\;.
\end{equation} 
The moment conditions (\ref{cond:Mr})$_r$ are required to control the growth of the graph, for instance to avoid the creation of traps for the random walk  (see, Theorem~\ref{thm:totti}).

The next two conditions are related to the\textit{ asymptotic behavior} of the environment process. For $n\ge 1$ 
define $S_n := \sum_{j=1}^n \xi_j$. 
  We say $\boldsymbol{\xi}$ satisfies assumption~(\ref{cond:S}) if there exists a positive constant $c$ and a function~$g: \N \setminus \{0\} \to \mathbb{R}^+$ of non-summable inverse ($\sum_{n =1}^{\infty}\frac{1}{g(n)} = \infty$) such that
\begin{equation}\tag{S}\label{cond:S}
P\left( \limsup_{n \rightarrow \infty} \frac{S_n}{g(n)} \le c\right) = 1\;.
\end{equation}
{In essence condition \eqref{cond:S} assures that the number of added leaves by time $n$ is bounded from above by a function  $g(n)$, with $g(n)/n \to 0$ (as $n$ grows), and thus the degree of a vertex cannot grow too fast. }

We say the environment $\boldsymbol{\xi}$ satisfies condition (\ref{cond:I}) if there exist a positive constant $c$ and a positive function~$f: \mathbb{N}\setminus \{0\}\to \mathbb{R}^+$ of summable inverse ($\sum_{n =1}^{\infty}\frac{1}{f(n)} < \infty$), such that
\begin{equation}\tag{I}\label{cond:I}
\begin{split}
P \left( \liminf_{n \rightarrow \infty} \frac{S_n}{f(n)} \ge c\right) & = 1\;. 
\end{split}
\end{equation}
{Condition \eqref{cond:I} requires that the number of added leaves by time $n$ is bounded from below by a  function $f(n)$, with $n/f(n)\to 0$ (as $n$ grows), and thus the degree of a vertex may grow very fast. }

%
We end this section with some important examples of environment for which one may verify some of the above conditions. The first one is the particular case where the environment $\boldsymbol{\xi}$ is i.i.d. with finite mean. It satisfies condition~(\ref{cond:UE}) and (\ref{cond:Mr})$_1$. Moreover, the Strong Law of Large Numbers assures that condition~(\ref{cond:S}) also  holds. Even more generally, if the environment  is an ergodic process, then (\ref{cond:S}) follows from the Ergodic Theorem.

On the other hand, for an independent environment such that $\xi_j$, $j\ge 1$, have heavy tails, then (\ref{cond:I}) holds. For instance, consider $\xi_j$ as independent random variables having power-law distributions such that
\[
P\left( \xi_j \ge x \right) \ge \frac{\delta}{x^{\alpha}}, \ \forall \, j \ge 1\;, 
\]
for some $\delta > 0$ and $\alpha \in (0,1)$. Now, consider $\beta\in(\alpha,1)$. Then for all $i$ large enough we get
$$
P \left(S_i \leq  i^{\frac{1}{\beta}}\right) 
\leq P\left(\max\{\xi_{1},\dots,\xi_{i}\}\leq i^{\frac{1}{\beta}}\right)
= \prod_{j=1}^i P\left(\xi_{j} \le i^{\frac{1}{\beta}}\right)
\le \left(1-\frac{\delta}{i^{\frac{\alpha}{\beta}}}\right)^i
\leq e^{-\delta \, i^{\gamma}}, 			
$$
where $\gamma = 1 - \frac{\alpha}{\beta}$. Since $\sum_{i\geq 1}e^{- \delta \, i^{\gamma}}<\infty$,  Borel-Cantelli lemma shows that
$$
P \left( \liminf_{i\to\infty} \frac{S_i }{i^{\frac{1}{\beta}}}\geq 1\right) = 1\;, 
$$
i.e., condition \eqref{cond:I} holds for $f(i) = i^{1/\beta}$ and $c=1$.

\subsection{Main results}\label{sec:results}
In this section we  present our results and  trace a parallel with their possible counterparts in the more classical theory of random walk on random environments. Before we can state properly our main results  we need to introduce some  definitions. As said before we want to study recurrence/transience and related properties for a $(\boldsymbol{\xi},s)$-\Name{}~$\{(T_n ,X_n)\}_{n \ge 0}$. Note that for independent environments $\boldsymbol{\xi}$ the process $\{(T_n ,X_n)\}_{n \ge 0}$ is markovian and  under (\ref{cond:UE}) it is always transient in the usual sense, since $T_n$ increases. {However, the process $\{X_n\}_{n \ge 0}$ is non-markovian and since $T_n$ increases, we need adequate definitions of recurrence and transience. With a slight abuse of terminology, when we say that the \Name{} process is either recurrent or transient, we will always be referring ourselves to the corresponding process $\{X_n\}_{n\geq 0}$ and not to the 
pair $\{(T_n ,X_n)\}_{n \ge 0}$.}  

Let $\{(T_n ,X_n)\}_{n \ge 0}$ be a \Name{} and $(T_0, x_0)$ its initial state. 
{We say that $z \in T_0$ is ${\it recurrent}$ if 
$$
\Pd_{T_0,x_0, s, \xi} \big( X_n = z \ \text{infinitely often} \big) = 1 \;.
$$
Since the graph structure is increasing, we also need to define recurrence for vertices that are added to the graph at some time.
So given a fixed realization $(\widetilde{T},x)$ of $(T_n,X_n)$ (i.e. $(\widetilde{T},x)$, with $x \in V(\widetilde{T})$ is any possible pair attainable from $(T_0,x_0)$ at time $n$) and $z \in V(\widetilde{T})$, we say that $z$ is ${\it recurrent}$ if
$$
\Pd_{\widetilde{T},x,s,\theta_n(\bxi)} ( X_n = z \ \text{infinitely often}) = 1\;,
$$
where and $\theta_k$ as the forward time shift such that~$\theta_k(\boldsymbol{\xi}) = \{\xi_{j+k}\}_{j \in \N}$. In the above display there is some abuse of notation if the environment is not independent, indeed $\Pd_{\tilde{T},x,s,\theta_k(\bxi)}$ depends on the whole history of the process until time $n$ and the display should be understood as the conditional probability
$$
\Pd_{T_0,x_0, s,\xi} \Big( X_n = z \ \text{infinitely often} \; \Big| T_n=\widetilde{T} , \; X_n = x \in V(\widetilde{T}), \; z \in V(\widetilde{T}) \Big) = 1 \;.
$$
The random walk in $(\boldsymbol{\xi},s)$-\Name{} is {\it recurrent} if, for every $n$ and every possible realization of $T_n$, all $z \in \cup_n V(T_n)$ are recurrent. 
If the \Name{} is not recurrent, we say that it is {\it transient}.}

\begin{remark} Recurrence or transience {for the $(\boldsymbol{\xi},s)$-\Name{} may depend on the choice of $T_0$, (even if $T_0$ is finite). Also, since the trees are connected, the \Name{} is irreducible in the usual sense that every vertex is reachable from any given configuration with positive probability. So irreducibility has no role in the results. }
\end{remark}
Let $ \eta_{z}$ denote the first time the random walk visits  vertex $z$, i.e.,   
\begin{equation}\label{eta}
\eta_{z} := \inf \left \lbrace n \ge 1 \; \middle | \;X_n = z \right \rbrace\;.
\end{equation} 
{Given $(T_0,x_0)$ and assuming the random walk in \Name{} is recurrent, we say that this random walk is {\it positive recurrent} if for any given realization of $(T_n,x_n)$ and $z \in V(T_n)$
$$
\mathbb{E}_{T_n,x_n,s,\theta_n(\bxi)}\left( \eta_z \right) < \infty\;.
$$
If the \Name{} is recurrent but not positive recurrent, then we say that it is {\it null recurrent}. This definition of positive recurrence is not directly comparable to the definition of positive recurrence for Markov chains. Since the tree is growing, there is no equilibrium and the mean return time to a vertex depends on the number of previous visits to  that vertex. It might increase on each visit to that vertex and even diverge to infinity as the vertex keeps being visited.}

\medskip 

As usual for trees, we will sometimes designate a particular vertex as the root of the tree, either because we simply want to fix a single vertex or because this vertex is special in some sense. We refer to this vertex simply as $\Root$.

We say that the \Name{} is {\it ballistic} if there exists a positive constant $c$, such that
	\begin{equation*}
	\liminf_{n \rightarrow \infty}\frac{\mathrm{dist}_{T_n}(X_n,\Root)}{n}  \geq c, \; \Pd_{T_0,x_0,s,\bxi}\text{- almost surely}\;.
	\end{equation*}
It is clear that every ballistic \Name{} is transient.

\medskip 
We can now state the main results of this paper.

\begin{theorem}[Recurrence/Traps for $s$ even] \label{rec-even} Consider a {$(s,\boldsymbol{\xi})$-\Name{}} process with $s$ even. For every initial state $(T_0,x_0)$ with $T_0$ finite, there exist two regimes:
\begin{enumerate}
\item[(i)]{\rm (Recurrence is inherited)} if $\boldsymbol{\xi}$ satisfies condition \eqref{cond:S}, then the {\Name{}} is recurrent.
\item[(ii)]{\rm (The dangerous environment)} if $\boldsymbol{\xi}$ is an independent environment  satisfying condition \eqref{cond:I}, then  
 there exists  $n$ such that the walker gets trapped at time $n$, $\Pd_{T_{0},x_0,s, \boldsymbol{\xi}}$-almost surely, i.e.
$$
\Pd_{T_{0},x_0,s, \boldsymbol{\xi}} \big( \textrm{there exists } x\in \cup_n V(T_n) \textrm{ and } k 
\textrm{ such that } X_{sn+k} = x \ \forall \, n \big) = 1\;.
$$
\end{enumerate}
\end{theorem}

For the specific case $s$ even and $\xi_j \equiv 1$, the recurrence of \Name{} was proved in \cite{FIN}; Theorem~\ref{rec-even} generalizes the result to much more general environments $\boldsymbol{\xi}$ and brings to light the possibility  for the walker to get trapped in some environments, i.e., localization phenomenon occurs, with the walker being locked in the neighborhood of a (random) vertex. We point out that Reinforced random walks also presents  the possibility to be either recurrent or to get trapped depending on the parameters of the model, see for instance \cite{D}.

The name for regime $(i)$ comes from the fact that under (\ref{cond:S}) and $s$ even the \Name{} process propagates recurrence, i.e., it is enough to have a single recurrent vertex and all vertices which are eventually added to the tree inherit recurrence from their parents. 

Although in Theorem~\ref{rec-even} we do not need any further assumption on $T_0$ other than finiteness, we will always assume  that $T_0$ has a self-loop at the $\Root$ when $s$ is even. This indeed makes things more interesting, since it is the case where {the height of the tree may increase}. This will be carefully discussed in Sections~\ref{sec:rec} and \ref{sec:height} (see, Proposition~\ref{prop:height}).

Theorem~\ref{rec-even} tells us that condition~\eqref{cond:S} guarantees the recurrence of every vertex. The next natural question regards the nature (null vs. positive) of this recurrence. 
As it turns out \eqref{cond:S}  alone does not guarantees neither null nor positive recurrence. Take, for instance,  a sequence $\xi_n\sim Ber(n^{-2})$, then the Borel-Cantelli lemma assures that condition~\eqref{cond:S} is satisfied and the tree will almost surely be finite, which implies that all its vertices will be positive recurrent. 
The next proposition provides sufficient condition for null recurrence. 

\begin{proposition}[Null recurrence for $s$ even]\label{prop:null-recurr}
Consider a $(s,\boldsymbol{\xi})-\Name$ process with $s$ even and independent environment $\boldsymbol{\xi}$  satisfying conditions  \eqref{cond:S} and  \eqref{cond:UE}. Then the {\Name{}} is null recurrent. 
\end{proposition}
We have another important result which is Theorem \ref{thm:totti}. We do not state it here to avoid excessive notation in this introduction. The theorem provides  conditions on the parameters for i.i.d. environments that imply  distinct local behaviors for the process. Specifically,  either the exit time from a vertex through a non-leaf neighbor has infinite mean, which immediately implies null recurrence, or it has finite mean and the \Name{} makes transitions between non leaves neighbors in finite mean times almost surely.

\medskip
When $s$ is odd the \Name{} process has a thoroughly different behavior. 

\begin{theorem}[Ballisticity for $s$ odd] \label{theo:ballistic}
If $s$ is odd and $\boldsymbol{\xi}$ is an independent environment satisfying \eqref{cond:UE}, then the {\Name{}} is ballistic. 
\end{theorem}
In \cite{FIN}, it is proved that the NRWW (\Name{} with $\xi_j \equiv 1$, for every $j$) is transient for~$s=1$. There, it is also conjectured that the NRRW is transient for every~$s$ odd. A first proof of ballisticity was given in \cite{FIORR} for the particular case with~$s=1$ and  the $\xi_j$ i.i.d. Bernoulli random variables. Our Theorem \ref{theo:ballistic} proves the conjecture of \cite{FIN} and  generalizes the result in \cite{FIORR}.

Let us draw a parallel with the classical theory of random walks on random environments. One of the main open problems in the theory of RWRE in $\Z^d$ is whether directional transience is equivalent to directional ballisticity for uniformly elliptic \textit{i.i.d.} environments. The conjecture involves two conditions each of them playing specific roles: the uniform ellipticity, which is a local condition, and the directional transience, which is a global one. The former prevents the walk from getting trapped in substructures of the space, whereas the latter tells us how the walk explores the environment in the long run (the interested reader may read more about ellipticity and ballisticity  in \cite{S}).
Interestingly, Theorem~\ref{theo:ballistic} reveals that
for $\Name{}$ process with $s$ odd,  uniform ellipticity is enough to obtain ballisticity. In other words, in this settings, a local condition is enough to drive the walk away from its initial position at linear speed.

\begin{remark}
As usual we can define continuous versions of the {\Name{}} by considering that the time between the $(n-1)$-th and $n$-th transitions are i.i.d. exponential random variables with parameter $\lambda_n > 0$, $n\ge 1$. It is standard to show that recurrence/transience of the continuous time version follows from the same property for the discrete time version. If $(\lambda_n)_{n\ge 1}$ is bounded away from $0$ and $\infty$, then the same holds for properties like ballisticity and null recurrence. It is also possible to have a null recurrent {\Name{}} that generates a positive recurrent continuous time random walk which happens if $(\lambda_n)_{n\ge 1}$ diverges to infinity sufficiently fast. These derivative results for continuous time do not bring novelty when compared to similar conclusions obtained for standard continuous time random walks.   
\end{remark}

The structure of the paper is the following: In Section~\ref{sec:longtimes} we prove that the time for the  random walk to go from one vertex to other one sufficiently far apart has infinite mean (regardless of the parity of $s$). There we introduce an auxiliary process (called generalized loop-process) which will also play an important role in the proof of ballisticity. 
{In Section~\ref{sec:rec} we prove Theorem \ref{rec-even} and other results for $s$ even, such as Theorem \ref{thm:totti}.  Section~\ref{sec:odd} is devoted to the proof of ballisticity of {\Name{}} for $s$ odd. In Section~\ref{sec:height} we  show that the height of the tree diverges to infinity as time goes to infinity (regardless of the parity of $s$) and we prove Proposition \ref{prop:null-recurr}.} Lastly,  Section \ref{sec:fincom} finishes the paper with a brief discussion on possible extensions of our results.

%% file: loop-process.tex

\section{Long hitting times} \label{sec:longtimes}

In this section we focus on understanding the hitting times of single vertices in the \Name.
More specifically, let us assume that the \Name{} process starts on $(T_0, x_0)$ and that~$z$ denotes a vertex at distance $\ell$ of $x_0$ in $T_0$. Recall the definition of the first hitting time $\eta_z$ from Equation~\eqref{eta}. The aim of this section is to study the distribution of $\eta_z$ and how it depends on $\ell$ in independent environments; the main results are Corollary~\ref{cor:etaybound} and Lemma~\ref{lem:infinite-expectation}.

In order to study $\eta_z$, we shall introduce  a simpler process, called {\em Generalized Loop Process}. 
This simpler process is a generalization of  the Loop process  which was introduced  in \cite{FIORR} in the context of the BGRW, which is the \Name{} for $s=1$ and $\boldsymbol{\xi}$  i.i.d.  sequence of Bernoulli's random variables. Herein, building on the same ideas, we generalized  such a process. 
Some of the results concerning the Generalized loop process are minor variations of the one presented in \cite{FIORR}. 

\subsection{Generalized Loop Process}\label{sec:loopprocess}
Roughly speaking, a generalized loop process on an initial graph $G$ is a random walk such that at each time $t=ms$ adds $\xi_{t}$ loops to its position  and then chooses uniformly one edge of its current position to walk on. Specifically, the number of vertices in the graph stays constant during the evolution of the process and it is equal to the number of vertices in the initial graph. 

Although we can define the loop process over any graph, we will treat only the case in which it is defined over a specific graph called \textit{backbone}. A finite graph $\mathcal{B}$ is a \textit{backbone} of length $\ell$ if~$\mathcal{B}$ is a path of length $\ell$ with a loop attached to its $(\ell + 1)$-th vertex and possibly to remaining vertices, see Figure~\ref{fig:backbone} below.
\begin{figure}[h]
\begin{tikzpicture}

\tikzset{every loop/.style={min distance=2cm,looseness=30}}
\tikzstyle{every state}=[scale=0.2,draw, fill]

\node[state] at (0, 0)   (0) {};
\node at (0, -0.5)   () {$0$};

\node[state] at (2,0)   (1) {};
\node at (2, -0.5)   () {$1$};

\node[state] at (4,0)   (2) {};
\node at (4, -0.5)   () {$2$};

\node[state] at (6,0)   (3) {};
\node at (6, -0.5)   () {$\cdots$};

\node[state] at (8,0)   (4) {};
\node[state] at (10,0)   (5) {};

\node[state] at (12,0)   (6) {};
\node at (12, -0.5)   () {$\ell$};

\draw[thick] (0) -- (1);
\draw[thick] (1) -- (2);
\draw[thick] (2) -- (3);
\draw[thick] (3) -- (4);
\draw[thick] (4) -- (5);
\draw[thick] (5) -- (6);

\draw[thick]   (1) edge [in=30,out=130,loop]  (1);

\draw[thick]   (1) edge [in=50,out=160,loop]  (1);

\draw[thick]   (2) edge [in=50,out=130,loop]  (2);

\draw[thick]   (4) edge [in=70,out=170,loop]  (4);
\draw[thick]   (4) edge [in=50,out=140,loop]  (4);
\draw[thick]   (4) edge [in=10,out=100,loop]  (4);

\draw[thick]   (6) edge [in=30,out=-30,loop]  (6);

\end{tikzpicture}
\caption{A \textit{backbone} of length $\ell$ }
	\label{fig:backbone}
\end{figure}
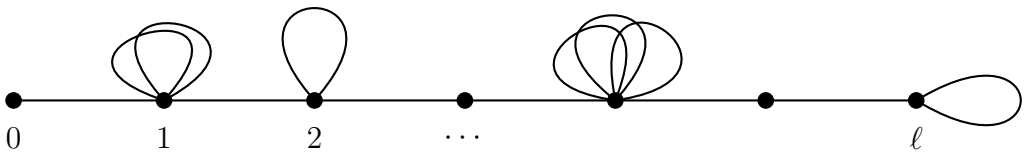

In this section, we denote by $\text{deg}_t(i)$  the number of edges attached to vertex $i$ at time $t$ counting loops only once. 
We refer to this quantity as \textit{degree of a vertex} even though we do not count loops twice.

The process depends on an integer greater zero $s$ and an independent environment process $\boldsymbol{\xi} = \{\xi_{n}\}_{n\in\N}$, which we assume satisfies condition \eqref{cond:UE}. We denote the Generalized loop process by $\{( \mathcal{B}_t ,X^{\text{loop}}_t)\}_{t \in \mathbb{N}}$ where $\mathcal{B}_t$ denotes the backbone at time $t$ and~$X^{\text{loop}}_t$ is one of its~$\ell+1$ vertices. Similarly to the \Name{}, the loop process is defined inductively according to the update rules presented below
\begin{enumerate}
	\item Generate $\mathcal{B}_{t+1}$ from $\mathcal{B}_{t}$ as follows:
	
	if $t= 0\,{\rm mod}\, s$ add   $\xi_{t}$  new loops to $X^{\text{loop}}_{t}$,
	
	if $t\neq 0\,{\rm mod}\, s$,  $\mathcal{B}_{t+1}=\mathcal{B}_{t}$.
	\item Choose uniformly one edge attached to $X^{\text{loop}}_t$ in $\mathcal{B}_{t+1}$. If the chosen edge is a loop, set $X^{\text{loop}}_{t+1}=X^{\text{loop}}_t$. Otherwise, $X^{\text{loop}}_{t+1}$ becomes the chosen $X^{\text{loop}}_t$ neighbor.
\end{enumerate}
Note the index $t+1$ of $\mathcal{B}$ on the rule~(2). This means we may add a loop at rule (1) and then select it at (2).


{Given a  backbone $\mathcal{B}$ of length $\ell$ and $i \in \{0,1,\cdots, \ell\}$ we denote by  $\Pd_{\mathcal{B},i}$ the law of the loop process when~$( \mathcal{B}_0 ,X^{\text{loop}}_0) \equiv (\mathcal{B},i)$, and by $\Ed_{\mathcal{B},i}$ the corresponding expectation. }
We will also let $\eta^{\text{loop}}_0$ be the first time $X^{\text{loop}}$ visits $0$,  i.e.,
\begin{equation*}
\eta^{\text{loop}}_0 := \inf\left \lbrace t \ge 0 \; \middle | \;X^{\text{loop}}_t = 0 \right \rbrace\;.
\end{equation*}
The next result provides  upper bounds for the cumulative distribution of $\eta^{\text{loop}}_0$.{
\begin{lemma}\label{lemma:xlooptime} Let $\boldsymbol{\xi}$ be an independent environment  satisfying condition \eqref{cond:UE}. Then there exist positive constants $c_1$, $c_2$ depending on $s$ and $\boldsymbol{\xi}$ such that, for all integer~$K\geq 1$ and for all $\beta \in (0,1)$
	\[
	\Pd_{\mathcal{B},\ell} \left( \eta^{\text{\rm loop}}_0 \le 
	e^{\beta K}\right) \le \exp\{-c_1K\} +  1-\left(1-\frac{1}{\ell}\right)^{c_2K} + \exp\{-(1-\beta)K\} \;.
	%
%
		\]
\end{lemma}}

\begin{proof}
It will be useful to look to loop process only when it actually moves from its position. Thus, define inductively the following
	\begin{equation*}
	\begin{split}
	& \tau_0 \equiv 0\;, \\
	& \tau_k := \inf\left \lbrace t > \tau_{k-1} \; \middle | \; X^{\text{loop}}_t \neq X^{\text{loop}}_{\tau_{k-1}}\right \rbrace\;.
	\end{split}
	\end{equation*}
We point out that the stopping times $\tau_k$ are not necessarily finite almost surely. If $\tau_k = \infty$ it means that the process gets trapped. As it will be shown in Theorem~\ref{rec-even}  (Section \ref{sec:rec}) under condition \eqref{cond:S} the times  $\tau_k$ are finite almost surely for all $k \ge 1$. However, when this is not the case, we can condition on the event that $X^{\text{loop}}$ is able to reach $0$  to estimate the probability in the statement, since on the complement of this event, we have $\eta_0^{\text{loop}} = \infty$. We leave the details to the reader and, henceforth, we suppose that $\tau_k$ is finite almost surely for all $k \ge 1$. This allows us to define the process
	\begin{equation*}
	Y_k := X^{\text{loop}}_{\tau_k}\;.
	\end{equation*}
	Note that by Strong Markov Property, $\{Y_k\}_k$ is a symmetric random walk on  the segment $[0, \ell] \cap \Z$, with reflecting barriers. We also define another stopping time
	\begin{equation*}
	\sigma :=  \inf\left\{ k > 0 \mid Y_k = 0\right\}\;,
	\end{equation*}
	and notice that $\eta^{\text{loop}}_0 = \tau_{\sigma}$.
	
	The idea of the proof is to show that the degree of vertex $\ell$ at time $\tau_{\sigma}$ is at least $e^K$ \textit{w.h.p.} which,  in turns,  guarantees that $\sum_{j=0}^{\lfloor\tau_\sigma/s\rfloor}\xi_{sj}$, i.e., the number of leaves added up to time $\tau_\sigma$,  is also at least $e^K$. Intuitively, having added an exponential number of leaves makes it harder for the walker to go back.
		In order to keep track of the degree of $\ell$, the following inequality (stochastic domination) will be useful  
	\begin{equation*}
	\mathrm{deg}_{\tau_n}(\ell) = 2+ \sum_{k=0}^{n-1} \mathbbm{1}\{Y_k = \ell\}\left( \sum_{ \substack{ m \in [\tau_k, \tau_{k+1}), \\ m \equiv 0\, {\rm mod}\,s}}\xi_m\right) \succeq  2 + \sum_{k=0}^{n-1} \mathbbm{1}\{Y_k = \ell\}\mathrm{Bin}\left(\left\lfloor \Delta \tau_k/s\right \rfloor, \kappa \right)\;,
	\end{equation*}
	where $\Delta \tau_k := \tau_{k+1} - \tau_k$ and $\kappa$ is the uniform ellipticity constant given by condition (\ref{cond:UE}). Note that when $Y_k = \ell$, the amount of time $X^{\text{loop}}$ takes to leave $\ell$ is $\Delta \tau_k$, which in turns satisfies
	\[
	\Delta \tau_k \succeq \mathrm{Geo}\left(1/\mathrm{deg}_{\tau_k}(\ell)\right)\;.
	\]
	Since the degree is non-decreasing, for any time $t \in [\tau_k, \tau_{k+1}]$, the probability of leaving $\ell$ given the past up to time $t-1$ is at most $1/\mathrm{deg}_{\tau_k}(\ell)$. Thus one may construct a coupling in such way that $\Delta \tau_k$ is greater than a geometric  random variable with parameter~$1/\mathrm{deg}_{\tau_k}(\ell)$. Thus, we have
	\begin{equation*}
	\mathrm{deg}_{\tau_n}(\ell) \succeq \sum_{k=0}^{n-1} \mathbb{1}\{Y_k = \ell\}\mathrm{Bin}\left(\left \lfloor \mathrm{Geo}\left( \frac{1}{\deg_{\tau_k}(\ell)}\right) /s\right \rfloor, \kappa\right)\;.
	\end{equation*}
	In order to avoid clutter, we simplify the notation defining $G_k := \mathrm{Geo}\left( 1/\deg_{\tau_k}(\ell)\right)$ and $d_k := \deg_{\tau_k}(\ell)$. With this notation, we claim that
\begin{claim} For all $\varepsilon \in (0,1)$, there exists a constant $q = q(s,\kappa, \varepsilon) \in (0,1)$ such that
	\begin{equation*}
	\Pd_{\mathcal{B},\ell}\left(Y_k = \ell, \mathrm{Bin}\left( \left \lfloor \frac{G_k}{s} \right \rfloor, \kappa\right) \ge  \varepsilon \kappa \frac{d_k}{s} \;  \middle |\; \mathcal{F}_{\tau_k} \right) \ge q\mathbb{1}\{Y_k = \ell\}\;, \quad \Pd_{\mathcal{B},\ell}-a.s.\;.
	\end{equation*}
\end{claim}

	\begin{claimproof}  By Chernoff bound, we have that
		\begin{equation*}
		\begin{split}
		&\Pd_{\mathcal{B},\ell}\left( Y_k = \ell, \mathrm{Bin}\left(\left \lfloor \frac{G_k}{s} \right \rfloor, \kappa \right) \ge \e \kappa  \frac{d_k}{s},\, G_k \ge  d_k \vee s  \middle | \mathcal{F}_{\tau_k}, G_k\right) \\
		&\ge \left(1- e^{-\frac{(1-\e)^2}{2}\kappa \left \lfloor\frac{G_k}{s}\right \rfloor}\right)\mathbb{1}\{Y_k = \ell, G_k \ge d_k \vee s  \} \\
		& \ge \left(1- e^{-\frac{(1-\e)^2}{2} \kappa }\right)\mathbb{1}\{Y_k = \ell, G_k \ge d_k \vee s \}\;.
		\end{split}
		\end{equation*}
		 Taking the conditional expectation with respect to $\mathcal{F}_{\tau_k}$ on the above inequality yields
		\begin{equation*}
		\begin{split}
		&\Pd_{\mathcal{B},\ell}\left( Y_k = \ell, \mathrm{Bin}\left(\left \lfloor \frac{G_k}{s} \right \rfloor, \kappa \right) \ge \e \kappa  \frac{d_k}{s}, G_k \ge  d_k \vee s \middle | \mathcal{F}_{\tau_k}\right)\\
		& \ge \left(1- e^{-\frac{(1-\e)^2}{2}\kappa}\right)\Pd_{\ell} \left( Y_k = \ell, G_k \ge d_{k} \vee s\middle | \mathcal{F}_{\tau_k} \right) \\
		& \ge \left(1- e^{-\frac{(1-\e)^2}{2}\kappa}\right)\left(1-\frac{1}{d_k}\right)^{d_k \vee s}\mathbb{1}\{Y_k = \ell\}\\
& \ge \left(1- e^{-\frac{(1-\e)^2}{2}\kappa}\right)\left(1-\frac{1}{d_k}\right)^{d_k  s}\mathbb{1}\{Y_k = \ell\}\\
		& \ge \underset{q(s,\kappa,\e)}{\underbrace{\left(1- e^{-\frac{(1-\e)^2}{2}\kappa}\right) e^{-\frac{3}{2}s}}}\mathbb{1}\{Y_k = \ell\}\;,
		\end{split}
		\end{equation*}
{where in the last inequality we used that $(1-\frac{1}{x})^x\geq e^{-\frac{3}{2}}$, $\forall x \geq 2$, and that  $d_k$ is greater than $2$ for all $k$.} 	
	\end{claimproof}
	
	By the claim, conditionally on the past, each time the process $\{Y_k\}_k$ visits $\ell$ it has a bounded away from zero probability of leaving $\ell$ with degree multiplied by $1+\e \kappa/s$. Thus, the degree of $\ell$ must be at least exponential in the number of visits of $\{Y_k\}_k$ to  $\ell$. 
		Now we are left to control the number of visits to $\ell$ by $Y$. To do this, let $N_{\sigma}(\ell)$ be the number of visits made by $Y$ to $\ell$ before time $\sigma$. Recall that $Y$ is a symmetric random walk on $\{0,1,\cdots,\ell\}$, thus $N_{\sigma}(\ell) \sim {\rm Geo}(1/{\ell})$. Moreover, the random variable $W$ that counts how many times we have successfully multiplied the degree of $\ell$ by $1+\e \kappa/s$ may be written as follows	
	\begin{equation*}
	W := \sum_{k=0}^{\sigma}\mathbb{1}\{Y_k = \ell\}\mathbb{1}\left\{\mathrm{Bin}\left( \left \lfloor \frac{G_k}{s} \right \rfloor, \kappa\right) \ge  \varepsilon \kappa \frac{d_k}{s}\right\}\;,
	\end{equation*}
	and dominates a random variable distributed as $\text{Bin}(N_{\sigma}(\ell),q)$. Thus, for any $K\ge 0$ 
	\begin{equation}\label{ineq:w}
	\begin{split}
	\Pd_{\mathcal{B},\ell}\left(W \le \frac{K}{\log(1+\e \kappa/s)}\right) & \le \Pd_{\mathcal{B}, \ell}\left(\mathrm{Bin}(N_{\sigma}(\ell),q) \le \frac{K}{\log(1+\e \kappa/s)}\right) \\
	& \le \Pd_{\mathcal{B},\ell}\left(\mathrm{Bin}(N_{\sigma}(\ell),q) \le \frac{K}{\log(1+\e \kappa/s)}\; \middle |\; N_{\sigma}(\ell) > \frac{2K}{q\log(1+\e \kappa/s)} \right) \\
	& + \Pd_{\mathcal{B},\ell}\left(N_{\sigma}(\ell) \leq 2q^{-1}K/\log(1+\e \kappa/s)\right) \\
	& \le \exp\{-c_1K\} + 1-\left(1-\frac{1}{\ell}\right)^{c_2K},
	\end{split}
	\end{equation}
with $c_2= \frac{2}{q \log(1+\e\kappa/s)}$ and $c_1=\frac{1}{4\log(1+\e\kappa/s)}$ obtained using Chernoff bound. 
{
Finally, observe that if~$W \ge K/\log(1+\e \kappa/s)$, then $\mathrm{deg}_{\tau_{\sigma}}(\ell) \ge 2e^{K}$. Thus, \eqref{ineq:w} yields
\begin{equation}\label{ineq:degless}
\Pd_{\mathcal{B}, \ell} \left( \mathrm{deg}_{\tau_{\sigma}}(\ell) \le 2e^{K}\right) \le \exp\{-c_1K\} + 1-\left(1-\frac{1}{\ell}\right)^{c_2K}.
\end{equation}
Given $K$, let $\eta_*$ be the following stopping time:
\begin{equation*}
\eta_* := \inf \left \lbrace n \ge 0 \; : \;  \mathrm{deg}_n(\ell) \ge 2e^{K}\right \rbrace\;.
\end{equation*}

By the Strong Markov Property applied to the loop process it follows that for any $\beta \in (0,1)$
\begin{equation}\label{ineq:etatau}
	\begin{split}
		\Pd_{\mathcal{B}, \ell}\left( \tau_\sigma \le e^{\beta K}, \eta_* < \tau_\sigma \right) & \le \Ed_{\mathcal{B}, \ell }\left(\mathbb{1}\{ \eta_* < \tau_\sigma\}\Pd_{\mathcal{B}_{\eta_*}, \ell }\left( \tau_\sigma \le e^{\beta K}\right)\right) \;.
	\end{split}
\end{equation}
In the event $\{\eta_* <\tau_{\sigma}\}$, the degree of vertex $\ell$ in the backbone $\mathcal{B}_{\eta_*}$ is at least~$2e^{K}$ and under such condition $\tau_{\sigma}$ dominates a geometric  random variable of parameter~$1/(2e^{K})$ which corresponds to  the time needed for the walker to visit vertex $\ell -1$. Thus, we obtain 
\begin{equation}\label{ineq:etatau2}
 \Ed_{\mathcal{B}, \ell }\left(\mathbb{1}\{ \eta_* < \tau_\sigma\}\Pd_{\mathcal{B}_{\eta_*}, \ell }\left( \tau_\sigma \le e^{\beta K}\right)\right) 
\le 1 - \left(1- \frac{1}{2e^{K}}\right)^{e^{\beta K}} \le \exp\{-(1-\beta)K\}\;.
\end{equation}
Finally using \eqref{ineq:degless},  \eqref{ineq:etatau} and \eqref{ineq:etatau2}  we obtain that
\begin{equation*}
	\begin{split}
		\Pd_{\mathcal{B}, \ell}\left( \tau_\sigma \le e^{\beta K} \right) & \le \Pd_{\mathcal{B}, \ell}\left( \tau_\sigma \le e^{\beta K}, \mathrm{deg}_{\tau_{\sigma}}(\ell) \ge 2e^{K} \right) + \Pd_{\mathcal{B}, \ell}\left(  \mathrm{deg}_{\tau_{\sigma}}(\ell) \le 2e^{K} \right) \\
		& \le \exp\{-(1-\beta)K\} + \exp\{-c_1K\} + 1-\left(1-\frac{1}{\ell}\right)^{c_2K}\;,
	\end{split}
\end{equation*}
  proving the lemma.
}
%
\end{proof}
The next proposition tells us that we may couple the \Name{} and the Generalized Loop Process, GLP,  in such way that $\eta_{z}$ is greater than $\eta^{\text{loop}}_0$ almost surely. 
The proposition is a mere generalization of Proposition~4.4 in \cite{FIORR} and its proof is in line with the one given therein  and thence will be omitted. The reader may check it by just replacing $s=1$ by any~$s$ and taking into account that the environment process may be capable of adding more than one leaf at once. 

\begin{proposition}[Coupling $\Name{}$ and the GLP; Proposition 4.4 in \cite{FIORR}]\label{prop:coupling} Let $T_0$ be a rooted locally finite tree, $x_0$ one of its vertices different from the root and $z$ an ancestor of $x_0$ at distance at least 2 from  $x_0$. Then, there exists a coupling of $\{(T_n,X_n)\}_{n\in\N}$ starting from~$(T_0,x_0)$ and a generalized loop process $\{(\mathcal{B}_n,X^{\text{\rm loop}}_n)\}_{n\in\N}$ starting from~$(\mathcal{B}(T_0,z,x_0),x_0)$ such that
	\[
	\Pd\left(\eta_{z}\ge \eta^{\text{\rm loop}}_0\right) = 1\;,
	\]
where $\mathcal{B}(T_0,z,x_0)$ is the backbone obtained from $(T_0,x_0)$ by the following procedure: i) remove all vertices in $T_0$ at distance greater than $2$ from the unique path connecting $x_0$ to $z$; ii) identify the remaining  vertices at distance one from the path with their neighbors on the path and consequently, all remaining edges not belonging to the path turn into loops.
\end{proposition}
We combine the bound given by Lemma \ref{lemma:xlooptime} with the above proposition to obtain an upper bound for the cumulative distribution of $\eta_{z}$ for a far enough $z$.
\begin{corollary}\label{cor:etaybound} Consider a $\Name{}$ started at $(T_0,x_0)$, with $T_0$ a rooted tree, $x_0$  a vertex different from the root (at distance at least $\ell$ from the root) and let $z$ be the ancestor of $x_0$ at distance $\ell$. Moreover, assume that $\boldsymbol{\xi}$ satisfies condition \eqref{cond:UE}.  Then,  there exists a positive constant $C$ depending on $s$ and $\boldsymbol{\xi}$ only, such that
	\[
	\Pd_{T_0,x_0,s, \boldsymbol{\xi}} \left( \eta_{z} \le e^{\sqrt{\ell}}\right) \le \frac{C}{\sqrt{\ell}}\;.
	\]
\end{corollary}
\begin{proof} By choosing $K = 2\sqrt{\ell}$ and $\beta = 1/2$ on Lemma \ref{lemma:xlooptime} we obtain that under condition (\ref{cond:UE}) it holds
		\[
	\Pd_{\mathcal{B}, \ell } \left( \eta^{\text{loop}}_0 \le e^{\sqrt{\ell}}\right) \le \frac{C}{\sqrt{\ell}}\;,
	\]
for some positive $C$ depending on the environment process and $s$ only. Finally, using the coupling given by Proposition~\ref{prop:coupling} the result follows.
\end{proof}

\subsection{Infinite expectation for the hitting time of far away vertices}
In this section, building on the previous (specifically on Lemma~\ref{lemma:xlooptime})  we prove that the hitting time of sufficiently far away vertices has an infinite expectation.
 This immediately implies  that the \Name{} process is either transient or null-recurrent.

\begin{lemma}\label{lem:infinite-expectation}
Let $\boldsymbol{\xi}$ be an independent environment  satisfying condition \eqref{cond:UE} and
$\{(T_n,X_n)\}_{n \in \N}$ a $(\boldsymbol{\xi},s)$-\Name{}. Then there exists $\ell_0=\ell_0(s, \kappa)$ such that
	\[
\mathbb{E}_{T_0,x, s, \xi}\left(\eta_z\right)=\infty\;,
	\]
	for all vertices $x, z$ such that $\dist{0}{x}{z}\geq \ell_0$.
\end{lemma}
\begin{proof} The result will follow from our upper bound for the cumulative distribution of $\eta_{z}$ given in Corollary \ref{cor:etaybound}.
{
\begin{align*}
\mathbb{E}_{T_0,x, s, \xi}\left(\eta_z\right)&= \sum_{k=0}^\infty\Pd_{T_0,x, s, \xi}(\eta_z\geq k )\geq 
(e^{\frac{1}{3}} -1) 
\sum_{m=0}^\infty \Pd_{T_0,x}(\eta_y\geq e^{m/3 } ) e^{ (m-1)/3}
\\
&\geq (e^{\frac{1}{3}} -1)
e^{-1/3}\sum_{m=0}^\infty e^{ m/3 }
\left(
\left(1-\frac{1}{\ell}\right)^{c_2m} - e^{-c_1m} - e^{-2/3m}\right)
\\
&= (e^{\frac{1}{3}} -1)
e^{-1/3}\left( 
\sum_{m=0}^\infty e^{ m/3 }
\left(1-\frac{1}{\ell}\right)^{c_2 m}  - 
\sum_{m=0}^\infty e^{m/3} e^{-c_1m} - 
\sum_{m=0}^\infty e^{ m/3 } e^{-2m/3}
\right)\;,
\end{align*}
}
where in the last inequality we used Lemma~\ref{lemma:xlooptime} and the coupling in Proposition~\ref{prop:coupling}. Clearly the last summation on the right-hand side converges and we can ignore this term. As regards  the second summation, recalling from Lemma~\ref{lemma:xlooptime} that $c_1=\frac{1}{4\log(1+\e\kappa/s)}$ with $\e \in (0,1)$ arbitrary, we can choose $\e$ sufficiently small (depending on $s$ and $\kappa$) such that $c_1>1/3$. This guarantees that the second summation also converges. In order to prove the claim we are left with showing that the first summation diverges. 
{
Using that $(1-\frac{1}{\ell})^{\ell}\geq e^{-3/2}$ for $\ell\geq 2$
we obtain that
\begin{align*}
\sum_{m=0}^\infty e^{ m/3 }
\left(1-\frac{1}{\ell}\right)^{c_2 m}\geq \sum_{m=0}^\infty e^{m \left(\frac{1}{3} - \frac{3c_2}{2\ell} \right) }\;.
\end{align*}
}
Recall (from Lemma~\ref{lemma:xlooptime}) that $c_2= \frac{2}{q \log(1+\e\kappa/s)}$; thus, given $s, \e, \kappa$ and $q$ it is possible to choose~$\ell$ sufficiently large such that $1/3-3c_2/2\ell>0$, which implies that the summation diverges.  
\end{proof}

%% file: recurrence-s-even-Rodrigo.tex

\section{The case $s$ even: Recurrence vs getting Trapped  } \label{sec:rec}

In this section we study the \Name{} when the step parameter is even for finite initial trees.  We shall assume that the initial tree $T_0$ (finite) has a particular vertex $r$, called the {\em root} of the tree,  which is the {\em unique} vertex with a self-loop.

As discussed in \cite{FIN}, the self-loop  at the root plays a prominent role in \Name{} when $s$ is even. Let us recall a few concepts to better understand the impact of  this local feature at the root.  
Let the level of a vertex be its distance (graph distance) from the root and define the level of the walker at time $n$ as $\dist{n}{X_n}{ {\rm root}}$. We say that the walker {$X$} at time $n$ is even (resp., odd) if $\dist{n}{X_n}{ {\rm root}}+n$ is even (resp., odd). Note that,  whenever the walker is even (resp., odd) new leaves can only be added to vertices with  even (resp., odd) levels.
In order for new leaves to be added to vertices whose levels have different parity, the walker must ``change its parity''. As it turns out, the walker can chance its parity only if it traverses the self-loop; indeed, this is the only case in which the distance from the root stays constant and the time increases by one (see, Figure~\ref{fig:walker-parity}). 
\begin{figure}[h]
\begin{tikzpicture}

\tikzset{every loop/.style={min distance=10mm,looseness=10}}
\tikzstyle{every state}=[scale=0.3,draw, fill]

\node at (-6, 0)   (a) {(0)};
\node at (-6, -1)   (b) {(1)};
\node at (-6, -2)   (c) {(2)};
\node at (-6, -3)   (d) {(3)};
\node at (-6, -4)   (e) {(4)};

\node at (6, 0)   (A) {} ;
\node at (6, -1)   (B) {};
\node at (6, -2)   (C) {};
\node at (6, -3)   (D) {};
\node at (6, -4)   (E) {};

\draw[thin, dotted] (a) -- (A);
\draw[thin, dotted] (b) -- (B);
\draw[thin, dotted] (c) -- (C);
\draw[thin, dotted] (d) -- (D);
\draw[thin, dotted] (e) -- (E);

\node[state,gray] at (0, 0)   (1) {};
\node[scale=0.4,draw, fill, black,circle] at (0, 0)  {};
\node at (0.8, 0.4)   () {root};
\node[state, gray] at (-2, -1)   (2) {};
\node[state, gray] at (0, -1)   (3) {};
\node[state, gray] at (2, -1)   (4) {};
\node[state] at (-3, -2)   (5) {};
\node[state] at (-1.5, -2)   (6) {};
\node[state] at (0, -2)   (7) {};
\node[state] at (1.5, -2)   (8) {};
\node[state] at (2.5, -2)   (9) {};
\node[state] at (3.5, -2)   (10) {};
\node[state, gray] at (-4, -3)   (11) {};
\node[state, gray] at (-2.5, -3)   (12) {};
\node[state, gray] at (-1, -3)   (13) {};
\node[state, gray] at (1, -3)   (14) {};
\node[state, gray] at (2.5, -3)   (15) {};
\node[state, gray] at (3.5, -3)   (16) {};
\node[state, gray] at (4.5, -3)   (17) {};
\node[state] at (-2.5, -4)   (18) {};
\node[state] at (-1.5, -4)   (19) {};
\node[state] at (-0.5, -4)   (20) {};
\node[state] at (1, -4)   (21) {};
\node[state] at (2.5, -4)   (22) {};
\node[state, red, rectangle] at (3.5, -4)   (23) {};


\draw[thick]   (1) edge [in=50,out=130,loop]  (1);

\draw[thick] (1) -- (2);
\draw[thick] (1) -- (3);
\draw[thick] (1) -- (4);
\draw[thick] (2) -- (5);
\draw[thick] (3) -- (6);
\draw[thick] (3) -- (7);
\draw[thick] (3) -- (8);
\draw[thick] (4) -- (9);
\draw[thick] (4) -- (10);
\draw[thick] (5) -- (11);
\draw[thick] (5) -- (12);
\draw[thick] (7) -- (13);
\draw[thick] (7) -- (14);
\draw[thick] (9) -- (15);
\draw[thick] (10) -- (16);
\draw[thick] (10) -- (17);
\draw[thick] (12) -- (18);
\draw[thick] (13) -- (19);
\draw[thick] (13) -- (20);
\draw[thick] (14) -- (21);
\draw[thick] (15) -- (22);
\draw[thick] (16) -- (23);
\end{tikzpicture}
\caption{If the random walk {$X$} at an even time is in the squared vertex, whose level is $(4)$, then the walker {$X$} is even and new leaf can only be added to vertices with even level, unless the walker uses the self-loop at the root. }
\label{fig:walker-parity}
\end{figure}
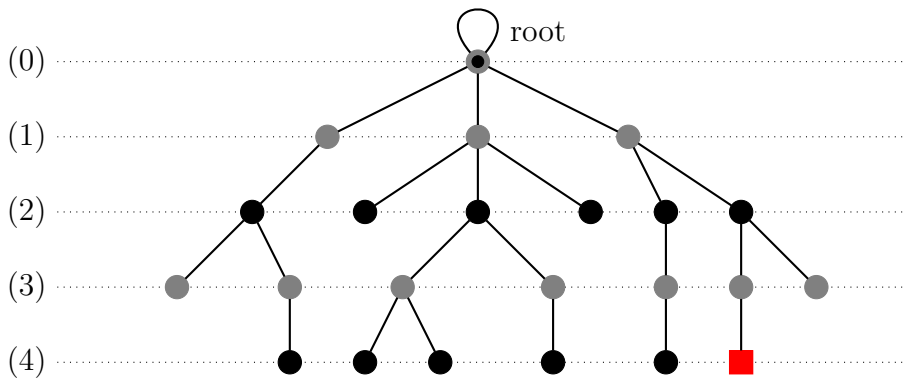
The change of parity of the walker imposes  crucial constraints on the growth of the tree when $s$ is even. 
\begin{itemize}
\item  The tree can grow to deeper levels only if the walker {$X$} changes its parity. Specifically,   if we consider a leaf $i$ added at time $t=ms$, subsequent leaves can be added to $i$ only if the walker changes its parity after time $t$.
\item If the walker {$X$} does not change its parity, new leaves can only be added to a finite set of vertices (those whose levels have  the same parity of the walker).
\end{itemize}
In most of the results presented in this chapter, we do not need the hypothesis of independence on the environment, and thus the Markov property. The reader should keep in mind that although the \Name{} is not necessarily markovian, if we only observe the evolution of the walker {$X$} on vertices with opposite parity as that of itself between consecutive uses of the self-loop, then this evolution is markovian. 

Following \cite{FIN}, a few questions naturally arise:  will the random walk change its parity an infinite number of times with probability one? What is the impact of  the environment~$\{\xi_{j}\}_{j \in \N}$ on the behavior of \Name{}? 
Note that, if the random walk changes its parity a finite number of time with positive probability then with positive probability the tree will have a finite depth.
A necessary condition to change parity an infinite number of times almost surely is that the walker visits the root an infinite number of times with probability one, i.e., that the root is recurrent. In this section we  show that:

\begin{itemize}
\item If  the  environment $\{\xi_{j}\}_{j \in \N}$ satisfies assumption \eqref{cond:S} (see, Section~\ref{sec:environment}) then:
\begin{enumerate}[i)]
\item the recurrence of the root is also a sufficient condition to assure the walker changes its parity infinitely often almost surely (Corollary~\ref{cor:recur}).
\item every vertex of the tree (also the root)  is recurrent (Theorem~\ref{rec-even}).
\end{enumerate}
\item If the environment satisfies condition \eqref{cond:I} (see, Section~\ref{sec:environment}) then the walker gets trapped almost surely, i.e., will keep on bouncing from one (random) vertex to its neighbors and back forever (Theorem~\ref{rec-even}). 
\end{itemize} 
Let us mention that,  for the specific case $(2k,1)$-\Name{} the recurrence is proved in \cite{FIN} and that  $(2k,1)$-\Name{} trivially satisfies condition~\eqref{cond:S}.

\medskip 
Before proving the main theorem of this section  we introduce some auxiliary results. The first one uses the fact that the random walk in \Name{} is symmetric to conclude that if the  walker visits a vertex $x$ an infinite number of times and traverse a specific edge incident to this vertex an infinite number of times, then it must traverse every edge incident to $x$ an infinite number of times. 
\begin{lemma}\label{lem:neighbors-recurrence}
	 Consider the {\Name{}} process $\{(T_n,X_n)\}_{n\in \N}$ with initial condition $(T_0,x_0)$. If there exist $x, y \in V(T_0)$ such that
	\begin{align*}
	\sum_{n=1}^\infty \mathbb{1}\{X_n=x\}\mathbb{1}\{X_{n+1}=y\}=+\infty\;, \quad \Pd_{T_0,x_0, s, \bxi}-a.s. \;, 
	\end{align*}
	then,  for any $z$ such that $\{x,z\} \in E(T_0)$ ($z$ neighbor of $x$)  
		\begin{align*}
	\sum_{n=1}^\infty \mathbb{1}\{X_n=x\}\mathbb{1}\{X_{n+1}=z\}=+\infty\;, \quad \Pd_{T_0,x_0, s, \bxi}-a.s.\;. 
	\end{align*}
\end{lemma}

\begin{proof}
If $\sum_{n=1}^\infty \mathbb{1}\{X_n=x\}\mathbb{1}\{X_{n+1}=y\}=+\infty$  $\Pd_{T_0,x_0, s, \bxi}$-a.s. then clearly the vertex $x$ is recurrent. This in particular, implies that the time of the $k$-th visit to $x$, i.e.,  $$\tau_k:=\inf\{n\geq \tau_{k-1}: X_n=x\}\;,$$ is finite almost surely. Thus,
\[
\sum_{n=1}^\infty \mathbb{1}\{X_n=x\}\mathbb{1}\{X_{n+1}=y\}=\sum_{k=1}^\infty \mathbb{1}\{X_{\tau_k}=x\}\mathbb{1}\{X_{\tau_k+1}=y\}
=\sum_{k=1}^\infty \mathbb{1}\{X_{\tau_k+1}=y\}=+\infty\;.
\]
Using a coupling argument, it can be shown that the distribution of the random variable~$\mathbb{1}\{X_{\tau_k+1}=y\}$ only depends on the degree of vertex $x$ at time $\tau_k+1$, and does not depend on the specific neighbors. Therefore, we conclude that 
\[
\sum_{k=1}^\infty \mathbb{1}\{X_{\tau_k+1}=y\}=\infty \quad \Pd_{T_0,x_0, s, \bxi}-a.s. \implies \sum_{k=1}^\infty \mathbb{1}\{X_{\tau_k+1}=z\}=+\infty\quad \Pd_{T_0,x_0, s, \bxi}-a.s.\;, 
\] 
for all $z$ neighbors of $x$.
\end{proof}
The second auxiliary result is fundamental to prove the main theorem of this section. It states that, under condition \eqref{cond:S} on the environment, the random walk does not get stuck  bouncing  from one vertex to its neighbors and back forever, whereas under condition \eqref{cond:I}, the walker has a positive probability to keep on bouncing back forever. Before stating the lemma, let us define 
\begin{equation*}
\tau_{\rm exit} := \inf \{ 2n \in \N \; : \; X_{2n} \neq  X_0\}\;,
\end{equation*}
 the first time the walker does not come back to the initial node after two steps.
\begin{lemma}\label{lem:time-exits}
	Consider an even {$(s,\boldsymbol{\xi})$-\Name{}} process. Then, for every initial state $(T_0,x_0)$ with $T_0$ finite:
	\begin{enumerate}[i)]
		\item if $\boldsymbol{\xi}$ satisfies condition \eqref{cond:S},  it holds that
		\[
		\Pd_{T_0,x_0, s, \bxi}\left(\tau_{\rm exit}(x_0) <\infty\right)=1\;,
		\]
		\item if $\boldsymbol{\xi}$ satisfies condition \eqref{cond:I}  and $x_0$ has at least a neighboring leaf in $T_0$ there exits a positive constant {$C$, depending only on $s$ and on the distribution of $\bxi$,} such that 
		\[
		\Pd_{T_0,x_0, s, \bxi}\left(\tau_{\rm exit}(x_0) =\infty\right)> e^{- C \,(\deg_{T_0}(x_0)-{\rm leaf}_{T_0}(x_0)) }>0\;,
		\]
		where $\deg_{T_0}(x_0)$ denotes the degree of $x_0$ in $T_0$ and ${\rm leaf}_{T_0}(x_0)$ the number of neighboring leaves. 
	\end{enumerate}
\end{lemma}

\begin{proof} 
\underline{Proof of item \textit{(i)}.} Let us first assume that $x_0$ is different from the root.  We shall prove the Lemma considering the ``worst'' possible scenario, i.e., the case in which $T_0$ is a star centered at $x_0$, whose degree is $d$ and  $d-1$ neighbors of $x_0$ are leaves and one neighbor is the root with a self-loop. We shall assume that $d\geq 2$, i.e., $x_0$ has at least a neighboring leaf; the case where $d=1$ is similar and easier. As it turns out, we shall prove that $\tau_{\rm exit}(x_0) <\infty$ almost surely, regardless of the value of $d$, which ``justify'' why this choice of $T_0$ is the worst possible scenario. Note that for this choice of $(T_0, x_0)$ the time $\tau_{\rm exit}(x_0)$ corresponds to the first time the walker traverses the self-loop. 
	To show that $\tau_{\rm exit}(x_0) <\infty$ almost surely, it is enough to show that $\widetilde{\tau}_{\rm exit}(x_0) <\infty$ a.s., where $\widetilde{\tau}_{\rm exit}$ denotes the first time the walker visits the root. This is because every time the walker visits the root, it has a constant probability (equal to 1/2) to traverse the self-loop. 

As long as $ \widetilde{\tau}_{\rm exit} > sn$, we have that  $S_{n}:= \sum_{j=0}^n \xi_j$ denotes the number of new leaves attached to $x_0$  up to time $sn$. Note that, if the random walk steps towards a leaf (not the root), it will necessarily be at $x_0$ in the next step.  The probability of choosing a leaf at time~$sn$ is $1-\frac{1}{d+S_n}$. Therefore, 
	\begin{eqnarray}\label{tailtau}
	\Pd_{T_0,x_0, s, \bxi}\left(\widetilde{\tau}_{\rm exit}> sn\right)
	& = &  E \big( \Pd_{T_0,x_0, s, \bxi}\left(\widetilde{\tau}_{\rm exit}> sn \big| \bxi \right) \big)
	\, = \, E \Big(\Pi_{i=0}^{n-1}\left(1-\frac{1}{d+S_i}\right)^{s/2}\Big) \nonumber \\
	& = & E \Big(\exp\big\{\frac{s}{2}\sum_{i=0}^{n-1}\log(1-\frac{1}{d+S_i})\big\}\Big)\;,		
	\end{eqnarray}
	where $\mathrm{E}$ denotes the expectation with respect to the environment $\boldsymbol{\xi}$.
	Since $\log(1-x)\leq -x$, for $0\le x <1$,   we obtain that  
	\begin{align*}
	\lim_{n\rightarrow\infty} \Pd_{T_0,x_0, s, \bxi}\left(\widetilde{\tau}_{\rm exit}> sn\right) & \le \lim_{n\rightarrow\infty}E \Big(\exp\big\{-\frac{s}{2}\sum_{i=0}^{n-1}\frac{1}{d+S_i}\big\}\Big)=E \Big(\exp\big\{-\frac{s}{2}\lim_{n}\sum_{i=0}^{n-1}\frac{1}{d+S_i}\big\}\Big),
	\end{align*}
	where the last inequality follows from the Dominated Convergence Theorem. Moreover, using that we are under condition \eqref{cond:S}, it follows
	\[
	\limsup_{i \rightarrow \infty}\frac{S_i}{g(i)} \le c, \; P\text{-almost surely}\;, 
	\]
	which implies that {there exists $L(\omega)$, for almost every $\omega$,} such that $S_i(\omega)< 2g(i)c$ for all $i\ge L(\omega)$. Then
	\begin{equation} \label{Sinfty}
	\sum_{i=L(\omega)}^{n-1}\frac{1}{d+S_i(\omega)}\geq \sum_{i=L(\omega)}^{n-1}\frac{1}{d+2g(i)c}\rightarrow \infty \text{ as } n\rightarrow\infty\;,
	\end{equation}
	regardless the value of $d$. Thus
	\begin{align*}
	\Pd_{T_0,x_0, s, \bxi}\left(\widetilde{\tau}_{\rm exit}=\infty\right)=\lim_{n\rightarrow\infty} \Pd_{T_0,x_0, s, \bxi}\left(\widetilde{\tau}_{\rm exit}> sn\right) \le E \Big(\exp\big\{-\frac{s}{2}\lim_{n}\sum_{i=0}^{n-1}\frac{1}{d+S_i}\big\}\Big)= 0\;.
	\end{align*}
	{Let us now consider the case $x_0={\rm root}$. We shall prove the Lemma considering the ``worst'' possible scenario, which in this case corresponds to having  $T_0$  a star
	centered at $x_0$ whose degree is $d$ (we can assume $d\geq 2$) and $d-1$ neighbors of 
	$x_0$ are leaves. Note that in this case $\tau_{\rm exit}(x_0)$ corresponds to the first time the walker traverses the self-loop (at an odd time) and in the subsequent step  visits a leaf.    
	To show that $\tau_{\rm exit}(x_0) <\infty$ almost surely, it is enough to show that $\widetilde{\tau}_{\rm exit} <\infty$ a.s., where $\widetilde{\tau}_{\rm exit}$ denotes the first time the walker uses the self-loop. This is because every time the walker traverses the self-loop at an odd time, it has a  probability bigger than or equal to $\frac{d-1}{d}$ to visit a leaf. Now the proof follows the same line of reasoning as above.}
	
	\medskip 
	
	\noindent
	\underline{Proof of item \textit{(ii)}.} Given a vertex $x_0\in V(T_0)$, we denote by $d=\deg_{T_0}(x_0)$  and 
	by~$\ell={\rm leaf}_{T_0}(x_0)$ its degree  and the number of neighboring leaves  in $T_0$, respectively. Note that, by the hypothesis we have that $\ell\geq 1$, while for the tree structure we have  $d-\ell\geq 1$. 
	Let us define $\widetilde{\tau}_{\rm exit}(x_0)$ the first time the random walk visits a non-leaf vertex neighbor of $x_0$. Then 
	\[
	\Pd_{T_0,x_0, s, \boldsymbol{\xi}} \left( \tau_{\rm{exit}}(x_0)> sn\right)\geq \Pd_{T_0,x_0, s, \boldsymbol{\xi}} \left( \widetilde{\tau}_{\mathrm{exit}}(x_0)> sn\right)\;,
	\]
	for every $s$ and $n$. Thus, to prove the claim it suffices to show that 
	$\Pd_{T_0,x_0, s, \bxi}\left(\widetilde{\tau}_{\rm exit}(x_0) =\infty\right)$ is strictly positive.

	Notice that, as long as $ \widetilde{\tau}_{\rm exit}(x_0) > sn$, the probability of choosing a leaf at time $sn$ is equal to~$1-\frac{d - \ell}{d+S_n}$. Therefore, similarly to Equation~\eqref{tailtau},  we have that 
	\begin{equation*}
	\Pd_{T_0, x_0, s, \boldsymbol{\xi}} \left(\widetilde{\tau}_{\rm exit}(x_0) > sn\right) =E\left( \exp \Big\{\frac{s}{2} \sum_{i=0}^{n-1}\log \Big(1- \frac{d-\ell}{d+S_i}
	\Big)\Big\}\right)\;,
	\end{equation*}
	where $E$ denotes the expectation with respect to the environment $\boldsymbol{\xi}$.  Recall that  environment condition \eqref{cond:I}  assures that
	\begin{equation} \label{I2c}
	{P}\left( \liminf_{n\to\infty} \frac{S_n}{f(n)}\geq 2c\right) = 1\;.
	\end{equation}
	Therefore, for every $0<\varepsilon<1$, we can find a measurable subset $\Omega_{\varepsilon}$ {and $n_{0} = n_0(\varepsilon)$} such that $P(\Omega^c_{\varepsilon})>\varepsilon$ and $S_i \geq cf(i)$ on $\Omega_{\epsilon}^{c}$, for all $i\geq n_0$. Hence for $n>n_0$ we get	
	\begin{align*}
	E\left(\exp \Big\{ \frac{s}{2} \sum_{i=0}^{n-1}\log \Big(1- \frac{d-\ell}{d+S_i}
	\Big)\Big\}\right) &\geq e^{\frac{s}{2} n_0\log\left( 1- \frac{d-\ell}{d}\right)}E\left(
	\exp\Big\{\frac{s}{2}\sum_{i=n_0}^{n-1}\log\Big(1-\frac{d-\ell}{d +cf(i)}\Big)\Big\}\mathbbm{1}_{\Omega_{\varepsilon}^c} \right)
	%
	\\
	&\ge \varepsilon e^{\frac{s}{2} n_0\log\left( 1- \frac{d-\ell}{d}\right)}\exp\left\{ \frac{s}{2}\sum_{i=n_0}^{n-1}\log\Big(1-\frac{d-\ell}{d+cf(i)}\Big)\right\}
	\;.
	\end{align*}
	Since $\log(1-x)\geq -x -\frac{x^2}{1-x}$,  for $0\leq x<1$, we have that  
	\begin{align*}
	\sum_{i=n_0}^{n-1}&\log\left(1-\frac{d-\ell}{d+cf(i)}\right)\geq -\sum_{i=n_0}^{n-1}\frac{d-\ell}{d+cf(i)} - \sum_{i=n_0}^{n-1}\frac{(d-\ell)^2}{\left(d+cf(i)\right)\left(\ell+cf(i)\right)}
	\\&
	=- \sum_{i=n_0}^{n-1} \frac{d-\ell}{\ell+cf(i)}\ge - \sum_{i=n_0}^{n-1} \frac{d-\ell}{cf(i)} \;.
	\end{align*}
	{Moreover $\log\left( 1- \frac{d-\ell}{d}\right) = \log(\ell) - \log(d) \ge - (d-\ell)$. }
	
	Using the hypothesis that $\sum_{i=1}^\infty \frac{1}{f(i)}<\infty$, we obtain that there exists a positive constant~$\widehat{C}$ such that 
	\begin{align*}
	\Pd_{T_0, x_0, s, \boldsymbol{\xi}} \left(\widetilde{\tau}_{\rm exit}(x_0) =\infty \right)=\lim_{n\to \infty}\Pd_{T_0, x_0, s, \boldsymbol{\xi}} \left(\widetilde{\tau}_{\rm exit}(x_0) > sn\right)\ge \varepsilon e^{-\frac{s}{2}n_0 (d-\ell) } e^{ -\frac{s}{2} (d-\ell)\, \widehat{C} }> 0\;.
	\end{align*} 
	{To finish the proof of the claim choose $C= - \log \epsilon + \frac{s}{2} (n_0 + \widehat{C})$ and note that $\varepsilon$ is fixed, $n_0$ depends on $\varepsilon$ and on the distribution of $\bxi$, and $\widehat{C}$ depends only on the distribution of $\bxi$. We point out that $n_0$ should increase with $\varepsilon$ and we have not considered the optional choice for $\varepsilon$ here.}
\end{proof}

\begin{remark}
	Note that in the above Lemma (part $ii)$), we implicitly use the fact that, at time $0$ the walker is in  $x_0$, and  new leaf can be added  to $x_0$ without the need to change parity. In particular, if the walker  steps on a vertex with the ``wrong'' parity,  it will not be able to add leaves to the vertex without changing its parity  first, and in particular not before the corresponding $\tau_{\rm exit}$, thus the Lemma will not be true in this case! 
	In the sequel we will need to use this Lemma together with the strong Markov property and therefore we must be sure the walker  steps on a vertex with the right parity.    
\end{remark}

\begin{remark} \label{geralSI}
Lemma~\ref{lem:time-exits} also holds if the environment $\bxi$ is shifted by an almost surely finite random time. First consider (i). If $\bxi$ satisfies condition \eqref{cond:S}  then the environment shifted by an almost surely finite random time also satisfies condition \eqref{cond:S}. To see that, it is enough to show that condition \eqref{cond:S}  is preserved by any finite shift of the environment. For $k\ge 0$ let $\theta_k$ be the forward time shift such that~$\theta_k(\boldsymbol{\xi}) = \{\xi_{j+k}\}_{j \in \N}$ and $S_n \comp \theta_k = \sum_{j=k+1}^{n+k} \xi_j = S_{n+k} - S_k$. If $\bxi$ satisfies condition \eqref{cond:S} with function $g$ and constant $c$, then for any $k\geq 0$, the environment $\theta_k(\boldsymbol{\xi})$ satisfies condition \eqref{cond:S} with function $\tilde{g}(n)=g(n+k)$ and constant $c$. Although $k$ might be random, we still have an almost surely divergence in \eqref{Sinfty} and (i) holds. Now we consider (ii) which is even simpler. Suppose that $\bxi$ satisfies \eqref{I2c}. Observe that $f(n) \rightarrow \infty$. Then
$$
\frac{S_n \comp \theta_k}{f(n)} = \frac{f(k+n)}{f(n)} \frac{S_{n+k}}{f(k+n)} - \frac{S_k}{f(n)} \ge c\;,
$$
for all sufficiently large $n$ depending on the possibly random $k$. This implies that the proof of (ii) is the same.
\end{remark}

For the subsequent results we will need to define a  random walk over an auxiliary graph that will play a crucial role in the proofs we are going to provide. 

Let us denote by $Y_n=X_{2n}$ the position of the walker after two steps, and define the sequence of stopping times: $\phi_0\equiv 0$ and for $k\geq 1$
\[
\phi_k:=\inf\{n>\phi_{k-1}: Y_n\neq Y_{\phi_{k-1}}\}\;.
\]
Note that, under condition \eqref{cond:S}, Lemma~\ref{lem:time-exits} guarantees  that $\phi_k$ is almost surely finite, for every $k$. Thus the process $\{Z_k\}_{k \in \N}$ defined as
\begin{equation}\label{def:Z}
Z_k := Y_{\phi_k}\;,
\end{equation}
is well defined.

Interestingly, as long as the walker { $X$}  does not traverse the self-loop, the process $Z$ is homogeneous. More specifically, if $t_m$ denotes the $m$-th time the walker { $X$} crossed the self-loop, 
then $Z_k$ for $t_m \leq k \leq t_{m+1}$ is a symmetric  random walk on a graph, whose structure only depends on the 
\Name{}  process up to time $t_m$, and remains fixed  during time $t_m$ and~$t_{m+1}$ (see, Figure~\ref{fig:Z-process}).

\begin{figure}[h]
\begin{tikzpicture}[scale=0.55]

\tikzset{every loop/.style={min distance=10mm,looseness=10}}
\tikzstyle{every state}=[scale=0.2,draw, fill]

\node[state,gray] at (0, 0)   (1) {};
\node[scale=0.4,draw, fill, black,circle] at (0, 0)  {};
\node at (1.3, 0.4)   () {root};
\node[state, gray] at (-2, -1)   (2) {};
\node[state, gray] at (0, -1)   (3) {};
\node[state, gray] at (2, -1)   (4) {};
\node[state] at (-3, -2)   (5) {};
\node[state] at (-1.5, -2)   (6) {};
\node[state] at (0, -2)   (7) {};
\node[state] at (1.5, -2)   (8) {};
\node[state] at (2.5, -2)   (9) {};
\node[state] at (3.5, -2)   (10) {};
\node[state, gray] at (-4, -3)   (11) {};
\node[state, gray] at (-2.5, -3)   (12) {};
\node[state, gray] at (-1, -3)   (13) {};
\node[state, gray] at (1, -3)   (14) {};
\node[state, gray] at (2.5, -3)   (15) {};
\node[state, gray] at (3.5, -3)   (16) {};
\node[state, gray] at (4.5, -3)   (17) {};
\node[state] at (-2.5, -4)   (18) {};
\node[state] at (-1.5, -4)   (19) {};
\node[state] at (-0.5, -4)   (20) {};
\node[state] at (1, -4)   (21) {};
\node[state] at (2.5, -4)   (22) {};
\node[state, red, rectangle] at (3.5, -4)   (23) {};


\draw[thick]   (1) edge [in=50,out=130,loop]  (1);

\draw[thick] (1) -- (2);
\draw[thick] (1) -- (3);
\draw[thick] (1) -- (4);
\draw[thick] (2) -- (5);
\draw[thick] (3) -- (6);
\draw[thick] (3) -- (7);
\draw[thick] (3) -- (8);
\draw[thick] (4) -- (9);
\draw[thick] (4) -- (10);
\draw[thick] (5) -- (11);
\draw[thick] (5) -- (12);
\draw[thick] (7) -- (13);
\draw[thick] (7) -- (14);
\draw[thick] (9) -- (15);
\draw[thick] (10) -- (16);
\draw[thick] (10) -- (17);
\draw[thick] (12) -- (18);
\draw[thick] (13) -- (19);
\draw[thick] (13) -- (20);
\draw[thick] (14) -- (21);
\draw[thick] (15) -- (22);
\draw[thick] (16) -- (23);
\end{tikzpicture}

\begin{minipage}{0.75\textwidth}
\begin{tikzpicture}[scale=0.7]

\tikzset{every loop/.style={min distance=10mm,looseness=10}}
\tikzstyle{every state}=[scale=0.27,draw, fill]

\node[state] at (0, 0)   (1) {};
\node at (0.8, 0.4)   () {root};
\node[state,gray] at (10, 0)   (1a) {};
\node at (10.8, 0.4)   () {root};

\node[state, gray] at (8, -1)   (2a) {};
\node[state, gray] at (10, -1)   (3a) {};
\node[state, gray] at (12, -1)   (4a) {};

\node[state] at (-3, -2)   (5) {};
\node[state] at (-1.5, -2)   (6) {};
\node[state] at (0, -2)   (7) {};
\node[state] at (1.5, -2)   (8) {};
\node[state] at (2.5, -2)   (9) {};
\node[state] at (3.5, -2)   (10) {};

\node[state, gray] at (6, -3)   (11a) {};
\node[state, gray] at (8.5, -3)   (12a) {};
\node[state, gray] at (9, -3)   (13a) {};
\node[state, gray] at (11, -3)   (14a) {};
\node[state, gray] at (12.5, -3)   (15a) {};
\node[state, gray] at (13.5, -3)   (16a) {};
\node[state, gray] at (14.5, -3)   (17a) {};

\node[state] at (-2.5, -4)   (18) {};
\node[state] at (-1.5, -4)   (19) {};
\node[state] at (-0.5, -4)   (20) {};
\node[state] at (1, -4)   (21) {};
\node[state] at (2.5, -4)   (22) {};
\node[state, red, rectangle] at (3.5, -4)   (23) {};


\draw[thick]   (1) edge [thick,dotted]  (1a);

\draw[thick] (1a) -- (2a);
\draw[thick] (1a) -- (3a);
\draw[thick] (1a) -- (4a);

\draw[thick] (2a) -- (3a);
\draw[thick] (3a) -- (4a);

\draw[thick] (2a) -- (11a);
\draw[thick] (2a) -- (12a);
\draw[thick] (11a) -- (12a);

\draw[thick] (3a) -- (13a);
\draw[thick] (3a) -- (14a);
\draw[thick] (13a) -- (14a);

\draw[thick] (4a) -- (15a);
\draw[thick] (4a) -- (16a);
\draw[thick] (4a) -- (17a);
\draw[thick] (15a) -- (16a);
\draw[thick] (16a) -- (17a);


\draw[thick] (1) -- (5);
\draw[thick] (1) -- (6);
\draw[thick] (1) -- (7);
\draw[thick] (1) -- (8);
\draw[thick] (1) -- (9);
\draw[thick] (1) -- (10);
\draw[thick] (6) -- (7);
\draw[thick] (7) -- (8);
\draw[thick] (9) -- (10);
\draw[thick] (5) -- (18);
\draw[thick] (7) -- (19);
\draw[thick] (7) -- (20);
\draw[thick] (7) -- (21);
\draw[thick] (9) -- (22);
\draw[thick] (10) -- (23);
\draw[thick] (19) -- (20);
\end{tikzpicture}
\end{minipage}
\caption{The process $Z$ and its graph corresponding to the situation depicted in Figure~\ref{fig:walker-parity}. As long as $Z$ (squared vertex) moves on the graph with black vertices, it will alter the gray graph structure. However, until {the process $X$}  crosses the self-loop (dotted line) {$Z$} will not see the gray structure. As soon as {$X$} crosses the self-loop the situation is interchanged.}
\label{fig:Z-process}
\end{figure}
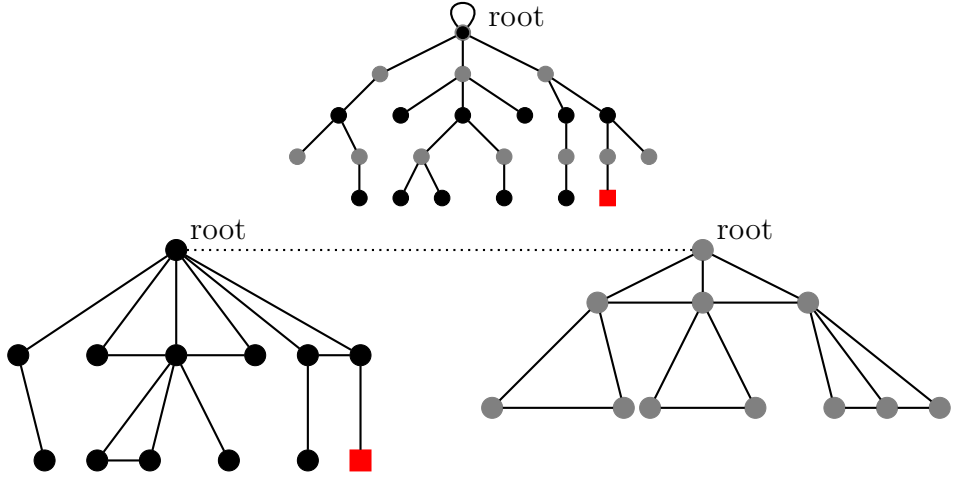
We say that the walker gets {\em trapped} at time $n$ if $\tau_{\rm exit}(X_n)=\infty$.
Note that, by the definition of $\tau_{\rm exit}$, if the walker gets trapped at time $n$ then it will also get trapped at time $n+2k$ for every $k$. 

\begin{remark}
If the walker {$X$} gets trapped at time $n$ with  probability one, then  all vertices at distance at least $2$ from $X_n$ will be clearly transient, whereas the vertex $X_n$ will be recurrent. As a matter of fact, also all the vertices at distance equal to $1$ from $X_n$ will be transient. Observe that if the walker gets trapped at time $n$ then it necessarily traverses  a finite number of times the edge connecting $X_n$ to its parent in the tree; the symmetry of the random walk assures that the same must hold for every edge incident to $X_n$. 
\end{remark}

We can now prove Theorem \ref{rec-even} the main result of this section. 

\medskip

\begin{proof}[Proof of Theorem~\ref{rec-even}]
\noindent \underline{Proof of item \textit{(i)}.} 
To prove the first part, it is enough to show that the walker traverses the self-loop an infinite number of times almost surely.  Indeed, whenever the latter happens, we know the root will be recurrent and,   using 
Lemma~\ref{lem:neighbors-recurrence}, we can conclude that the walker traverses an infinite number of times
any edge incident to the root. This will assure that the neighbors of the root will be recurrent. Knowing that these neighbors are recurrent and that the edges connecting them to the root are crossed an infinite number of times,  {  by Lemma~\ref{lem:neighbors-recurrence} we may conclude that all vertices at distance $2$ from the root are also recurrent. This  
 argument can then be iterated along all vertices of the tree.} 

\medskip 

Given $x, y$ vertices of the tree,  let us define 
$
J_n(x,y):=\sum_{k=1}^{n-1} \mathbb{1}\{X_k=x\}\mathbb{1}\{X_{k+1}=y\}
$. What we are after is to show that $\lim_n J_n(\Root, \Root)=\infty$ a.s..
\begin{claim}
$\lim_n J_n(\Root, \Root)=+\infty$, $\Pd_{T_0,x_0, s, \bxi}$-a.s. .
\end{claim}

{We present a proof of the claim just below defining a useful martingale and we emphasize that similar martingale techniques will also be used latter.} 


\medskip 
\begin{claimproof}
%
Without loss of generality, we can assume that there exists $y\in V(T_0)$ such that $\degr{0}{y}\geq 2$ and 
$\{\Root, y\}\in E(T_0)$. 

Let us define the following sequence of stopping times: $\sigma_0\equiv 0$ and for all $m\ge 1$ we define $\sigma_m:=\inf\{n>\sigma_{m-1}: X_{n-1}=X_n \}$ if $\sigma_{m-1}<\infty$ and $\sigma_m=\infty$, otherwise.
Since we must have $X_{\sigma_m} = \Root$, to prove the claim it is enough to show that 
the event $\{\exists m : \sigma_m=\infty\}$ has probability $0$ with respect to  $\Pd_{T_0,x_0, s, \bxi}$. 
Note that, for $y\in V(T_0)$ such that $\{\Root, y\}\in E(T_0)$, using the Doob's Decomposition Theorem, we may decompose $J_n(\Root,y)$ as
\begin{equation}\label{eq:doob2}
J_n(\Root ,y):= M_n(\Root,y) + A_n(\Root)
\;,
\end{equation}
where $M_n(\Root,y)$ defined as
\[
M_n(x,y) := \sum_{k=1}^{n-1} \left( \mathbb{1}\{X_k=\Root\}\mathbb{1}\{X_{k+1}=y\} - \frac{\mathbb{1}\{X_k=\Root\}}{\degr{k+1}{\Root}}\right),
\]
is a mean zero bounded increments martingale with respect to the filtration $\{ \mathcal{F}_{k}\}_{k \in \N}$, where~$\mathcal{F}_k$ is the $\sigma$-algebra generated by the { process $\{(T_n, X_n)\}_{n\geq 0}$ up to time $k$ together with the information of $\xi_k$}. This implies that $T_{k+1}$ is measurable with respect to $\mathcal{F}_k$ and that
\[
\Pd_{T_{0},x_0,s, \boldsymbol{\xi}}\left( X_{k+1} = y,  X_k = \Root \mid \mathcal{F}_k \right) =  \frac{\mathbb{1}\{X_k=\Root\}}{\degr{k+1}{\Root}}\;.
\]
Thus, $A_n(\Root)$ is the predictable component defined as
\[
A_n(\Root) :=\sum_{k=1}^{n-1} \frac{\mathbb{1}\{X_k=\Root\}}{\degr{k+1}{\Root}}\;.
\] 
Moreover, since $M_n$ has bounded increments, we may apply Theorem~$5.3.1$ in \cite{Du}, which guarantees that if we define the following two sets 
\begin{align*}
& C_{\Root, y}=\{\lim_n M_n(\Root, y)\text{ exists and it is finite}\}\;,
\\
& 
D_{\Root, y}=\{\limsup_n M_n(\Root, y)=\infty \text{ and } \liminf_n M_n(\Root, y)=-\infty\}\;,
\end{align*}
then, it holds that $\Pd_{T_0,x_0, s, \bxi}\left( C_{\Root, y} \cup D_{\Root, y} \right)=1$. Using the latter result, we conclude that to prove  that $\Pd_{T_0,x_0, s, \bxi}\left(\{\exists m : \sigma_m=\infty\}\right)=0$ it suffices to show that the event $F$ defined as
\[
F:= \{\exists m : \sigma_m=\infty\} \cap \left(C_{\Root, \Root} \cup D_{\Root, \Root} \right) \cap  \left(C_{\Root, y} \cup D_{\Root, y} \right)\;,
\]
has probability zero. In order to do so, we first observe that for any $\omega \in \{\exists m : \sigma_m=\infty\}$ the process $Z$ { (see, Equation~\eqref{def:Z})} eventually evolves on a fixed graph, {since $X$}  will traverse the self-loop a finite number of times. Then, on the trajectory $\omega$, { either  $Z$ visits all the vertices of this fixed graph an infinite number of times (first case) or there exist vertices on this finite graph which are visited only a finite number of times (second case). Thus we may write $\{\exists m : \sigma_m=\infty\}= \{\exists m : \sigma_m=\infty\}' \cup \{\exists m : \sigma_m=\infty\}''$, where 
$\{\exists m : \sigma_m=\infty\}'$ is the set of trajectories corresponding to the first case,  and $\{\exists m : \sigma_m=\infty\}''$ to the second case.   Since on~$\{\exists m : \sigma_m=\infty\}$ the process $Z$ is a simple random walk on a finite graph and consequently recurrent,  the  probability of $\{\exists m : \sigma_m=\infty\}''$ is zero. Thus, to prove  the claim it is enough to show that 
\[
F':= \{\exists m : \sigma_m=\infty\}' \cap \left(C_{\Root, \Root} \cup D_{\Root, \Root} \right) \cap  \left(C_{\Root, y} \cup D_{\Root, y} \right) = \emptyset\;.
\]
}
{ Let us assume, towards a contradiction, there exists $\omega \in F'$. Then $\omega \in \{\exists m : \sigma_m=\infty\}'$, and  along  this sample path the process $Z$ eventually 
evolves on a fixed graph and visits all the  vertices of this fixed graph an infinite number of times.} Therefore, due to the assumption that $y$ is a child of the root and $y$ has at least one child, we know that either $Z$ visits $y$ an infinite number of times, or it visits all the  neighbors of $y$ an infinite number of times. Both cases imply $J_n(\Root,y)(\omega)\nearrow \infty$.

As far as the very same $\omega$ is concerned, we also know that  $\omega \in C_{\Root, y} \cup D_{\Root, y}$, which implies that either $M_n(\Root, y)(\omega)\to L(\omega)<\infty$ or $\liminf_n M_n(\Root, y)(\omega)=-\infty$. Knowing that $J_n(\Root,y)(\omega)\nearrow \infty$ and considering the possible cases for $M_n(\Root,y)(\omega)$ it follows that we must have that $A_n(\Root)(\omega)\nearrow \infty$. In the first case, this follows from Equation~\ref{eq:doob2},  whilst in the second it follows from the definition of $M_n(\Root, y)$. 

{ To show the contradiction we now show that if $\omega \in  \{\exists m : \sigma_m=\infty\}'$ then $A_n(\Root)(\omega)$ converges to a finite limit. Indeed, if $\omega \in  \{\exists m : \sigma_m=\infty\}'$ then by the definition of the stopping times $\sigma_m$ it follows that $J_n(\Root, \Root)(\omega)\nearrow K(\omega)<\infty$. Using Equation~\ref{eq:doob2} and  that $A_n(\Root)$ is positive, we have that $\limsup_n M_n(\Root, \Root)(\omega)\leq K(\omega)$. Given that the trajectory $\omega$ belongs to $ C_{\Root, \Root} \cup D_{\Root, \Root}$, 
it must be the case that the martingale converges to a finite limit, i.e., $M_n(\Root, \Root)(\omega)\to K'(\omega)<\infty$. This will imply that $A_n(\Root)(\omega)\nearrow K''(\omega)<\infty$, and we reach a contradiction.}  
%
%
%
%
%
%
%
\end{claimproof}

\bigskip

\noindent \underline{Proof of item \textit{(ii)}.} 
We need to show that $\Pd_{T_{0},x_0,s, \boldsymbol{\xi}} \left(\exists n\geq 0 :  \tau_{\rm exit}(X_{2n})=\infty \right)=1$. Due to the definition of $\tau_{\rm exit}(X_n)$ it is enough to show that 
\[
\sum_{n=1}^\infty \mathbb{1}\{  \tau_{\rm exit}(X_{2n})=\infty \}=+\infty\;, \quad \Pd_{T_0, x_0, \s, \boldsymbol{\xi}}-a.s. \;.
\]
In order to prove the above identity we argue in a similar way to the proof of item $i)$. In this case, we show first that if the predictable process given the the Doob's Decomposition Theorem converges to infinity, then the whole sum also goes to infinity. More formally, we prove the following claim first.
\begin{claim}
\begin{equation*}
\begin{split}
\sum_{k=1}^\infty \Ed_{T_0, x_0, \s, \boldsymbol{\xi}}\left(\mathbb{1}\{\tau_{\rm exit}(X_{\s k})=\infty\} | \mathcal{F}_{\s k} \right)=\infty, \text{ a.s.} & \implies \sum_{k=1}^\infty \mathbb{1}\{\tau_{\rm exit}(X_{\s k})=\infty\}=\infty, \text{ a.s.}\;.
\end{split}
\end{equation*}
\end{claim}

\begin{claimproof}
Let $R_n=\sum_{k=1}^n \mathbb{1}\{\tau_{\rm exit}(X_{\s k})=\infty\}$, we can write 
\[
R_n = M_n + A_n\;,
\]
where, 
\[
M_n :=\sum_{k=1}^n \bigg( \mathbb{1}\{\tau_{\rm exit}(X_{\s k})=\infty\}- \Ed_{T_0, x_0, \s, \boldsymbol{\xi}}\left(\mathbb{1}\{\tau_{\rm exit}(X_{\s k})=\infty\} |\mathcal{F}_{\s k} \right) \bigg)\;,
\] is a bounded increments martingale and 
\[
A_n :=\sum_{k=1}^n\Ed_{T_0, x_0, \s, \boldsymbol{\xi}}\left(\mathbb{1}\{\tau_{\rm exit}(X_{\s k})=\infty\} |\mathcal{F}_{\s k} \right)\;.\]
Using Theorem~$5.3.1$ in \cite{Du}, we know that either $\lim_n M_n$ exists and it is finite almost surely, or $\limsup_n M_n=\infty$ and $\liminf_n M_n=-\infty$ almost surely. Using the assumption that $\lim_n A_n=\infty$ almost surely, in both cases we can conclude that $\lim_n R_n=\infty$ almost surely. 
\end{claimproof}

\medskip 
{ Using the above claim, in order to prove item $ii)$ of Theorem~\ref{rec-even}, we
 are  left with showing that 
\begin{equation}\label{eq:sum}
\sum_{k=1}^\infty \Ed_{T_0, x_0, \s, \boldsymbol{\xi}}\left(\mathbb{1}\{\tau_{\rm exit}(X_{\s k})=\infty \}|\mathcal{F}_{\s k} \right)=\infty, \; \Pd_{T_0, x_0, \s, \boldsymbol{\xi}}-a.s.\;.
\end{equation}
The proof of \eqref{eq:sum} relies on  Lemma~\ref{lem:time-exits} (point $ii)$). However, when using Lemma~\ref{lem:time-exits} together with the Markov Property (which holds given the independence assumption on the environment $\boldsymbol{\xi}$) we need to account for how $\deg_{T_{sk}}(X_{sk})-{\rm leaf}_{T_{sk}}(X_{sk})$ grows with $k$. 
In order to do that, let us introduce an terminology: we say that a vertex is a {\em quasi-star} if it  has at least one neighboring leaf and  only one non-leaf neighbor. Specifically, 
$X_{n}$ is a quasi-star if ${\rm leaf}_{T_{n}}(X_n)\geq 1$ and  $\degr{n}{X_n}={\rm leaf}_{T_{n}}(X_n)+1$. 
%
Then,  we can  write
\begin{align}
\nonumber \Ed_{T_0, x_0, \s, \boldsymbol{\xi}}&\left(\mathbb{1}\{\tau_{\rm exit}(X_{\s k})=\infty\} |\mathcal{F}_{\s k} \right)=
\\ 
& \label{eq:first}\Ed_{T_0, x_0, \s, \boldsymbol{\xi}}\left(\mathbb{1}\{\tau_{\rm exit}(X_{\s k})=\infty\} \mathbb{1}\{ X_{sk} \text{ is a quasi-star} \}|\mathcal{F}_{\s k} \right) 
\\
& \label{eq:second}+  \Ed_{T_0, x_0, \s, \boldsymbol{\xi}}\left(\mathbb{1}\{\tau_{\rm exit}(X_{\s k})=\infty\} \mathbb{1}\{ X_{sk} \text{ is not a quasi-star}\}|\mathcal{F}_{\s k} \right) 
\end{align}

As far as \eqref{eq:first} is concerned, applying  the Markov Property, we obtain 
\begin{align*}
&\Ed_{T_0, x_0, \s, \boldsymbol{\xi}}\left(\mathbb{1}\{ X_{sk} \text{ is a quasi-star}\} \mathbb{1}\{\tau_{\rm exit}(X_{\s k})=\infty\} |\mathcal{F}_{\s k} \right)
\\
&=
\mathbb{1}\{ X_{sk} \text{ is a quasi-star}\} \Ed_{T_{\s k}, X_{\s k}, \s, \theta_{sk}(\boldsymbol{\xi})}\left(\mathbb{1}\{\tau_{\rm exit}(X_{0})=\infty\}\right)
\\
&
\geq\mathbb{1}\{ X_{sk} \text{ is a quasi-star}\} e^{-C}\;,
\end{align*}
where $\theta_{sk}$ denotes  the forward time shift  $\theta_{sk}(\boldsymbol{\xi}) = \{\xi_{j+sk}\}_{j \in \N}$, and in the last inequality we used that $X_{sk}$ is a quasi-star together with Lemma~\ref{lem:time-exits} (point $ii$). 
As regards \eqref{eq:second}, applying the Markov Property we have that 
\begin{align*}
&\Ed_{T_0, x_0, \s, \boldsymbol{\xi}}\left(\mathbb{1}\{ X_{sk} \text{ is not a quasi-star}\}\mathbb{1}\{\tau_{\rm exit}(X_{\s k})=\infty\} |\mathcal{F}_{\s k} \right)
\\
&=
\mathbb{1}\{ X_{sk} \text{ is not a quasi-star}\}\Ed_{T_{\s k}, X_{\s k}, \s, \theta_{sk}(\boldsymbol{\xi})}\left(\mathbb{1}\{\tau_{\rm exit}(X_{0})=\infty\}\right)\;.
\end{align*}
If $X_{sk}$ is not a quasi-star, $X_{sk}$ may have no neighboring leaves in $T_{sk}$ and thus Lemma~\ref{lem:time-exits} (point $ii$) does not necessarily apply. To be able to  apply Lemma~\ref{lem:time-exits},  we multiply by $\mathbb{1}\{ \xi_{sk}\geq 1\}$ and obtain that 
\begin{align*}
&\Ed_{T_0, x_0, \s, \boldsymbol{\xi}}\left(\mathbb{1}\{ X_{sk} \text{ is not a quasi-star}\}\mathbb{1}\{\tau_{\rm exit}(X_{\s k})=\infty\} |\mathcal{F}_{\s k} \right)
\\
&\geq 
\mathbb{1}\{ X_{sk} \text{ is not a quasi-star}\} \mathbb{1}\{ \xi_{sk}\geq 1\} \Ed_{T_{\s k}, X_{\s k}, \s,  \theta_{sk}(\boldsymbol{\xi})}\left(\mathbb{1}\{\tau_{\rm exit}(X_{0})=\infty\}\right)
\\
&
\geq\mathbb{1}\{ X_{sk} \text{ is not a quasi-star}\}\mathbb{1}\{ \xi_{sk}\geq 1\} \exp\left\{-C \left(
\deg_{T_{sk}}(X_{sk})-{\rm leaf}_{T_{sk}}(X_{sk}) \right) \right\}\;. 
\end{align*}
If $X_{sk}$ is not a quasi-star (but has at least a neighboring leaf), we do not have a straight forward  upper  bound for $\deg_{T_{sk}}(X_{sk})-{\rm leaf}_{T_{sk}}(X_{sk})$, uniformly in $sk$.  However, 
if define  
\[
c(T_{sk}):= \max_{x \in V(T_{sk})}\{\degr{sk}{x}- {\rm leaf}_{T_{sk}}(x)\}\;,
\]
it holds that  
\[
c(T_{sk}) \leq c(T_0) + \sum_{i=0}^k \mathbb{1}\{ X_{ik} \text{ is a quasi-star}\}\;.
\]
To see why the above inequality holds, note that when 
$\xi_{sk}$  new leaves are added to the tree, either $X_{sk}$ is a quasi-star (in $T_{sk+1}$) or not. If $X_{sk}$ is not a quasi-star (in $T_{sk+1}$) then the constant $c(T_{sk})$ does not change, since the degree of $X_{sk}$ in $T_{sk+1}$ has increased by $\xi_{sk}$ and all of them are leaves. If $X_{sk}$ is a quasi-star (in $T_{sk+1}$) then the constant $c(T_{sk})$ may have increased by at most one unit  (if the $\xi_{sk}$ new vertices were added to a leaf).

Overall, we obtain that 
\begin{align*}
&  \Ed_{T_0, x_0, \s, \boldsymbol{\xi}}\left(\mathbb{1}\{\tau_{\rm exit}(X_{\s k})=\infty\} |\mathcal{F}_{\s k} \right) \geq\mathbb{1}\{ X_{sk} \text{ is a quasi-star}\} e^{-C}
\\
& + \mathbb{1}\{ X_{sk} \text{ is not a quasi-star}\}\mathbb{1}\{ \xi_{sk}\geq 1\} \exp\left\{-C \left( c(T_0) + \sum_{i=0}^k \mathbb{1}\{ X_{ik} \text{ is a quasi-star}\} \right) \right\}
\;. 
\end{align*}
Note that, either $\sum_{k=1}^\infty \mathbb{1}\{ X_{sk}\text{ is a quasi-star}\} = +\infty$ or $\sum_{k=1}^\infty \mathbb{1}\{ X_{sk}\text{ is a quasi-star}\} =M< +\infty$. In the first case, clearly we obtain that $\sum_{k=1}^\infty \Ed_{T_0, x_0, \s, \boldsymbol{\xi}}\left(\mathbb{1}\{\tau_{\rm exit}(X_{\s k})=\infty\} |\mathcal{F}_{\s k} \right)= +\infty$, whereas  in the second case we obtain that 
\begin{align*}
\sum_{k=1}^\infty \Ed_{T_0, x_0, \s, \boldsymbol{\xi}}\left(\mathbb{1}\{\tau_{\rm exit}(X_{\s k})=\infty\} |\mathcal{F}_{\s k} \right)&\geq e^{-C \left(c(T_0) + M\right)}\sum_{k=1}^\infty \mathbb{1}\{\xi_{sk\geq 1}\} \mathbb{1}\{ X_{sk}\text{ is not a quasi-star}\}
\\
&
= e^{-C \left(c(T_0) + M\right)}\left(\sum_{k=1}^\infty \mathbb{1}\{\xi_{sk\geq 1}\} -M\right)\;.
\end{align*} 
The proof of \eqref{eq:sum} finally follows from noticing that condition \eqref{cond:I} implies 
\[
\sum_{k=1}^\infty \mathbb{1}(\xi_{\s k}\geq 1)=+\infty, \;\Pd_{T_0, x_0, \s, \boldsymbol{\xi}}-\text{a.s.}\;,
\] since whenever $S_n=\sum_{k=1}^n \xi_{\s k}$ converges to a finite limit, condition \eqref{cond:I} cannot hold.
}
\end{proof}
One consequence of item $(i)$ of  { Theorem~\ref{rec-even}} is the following corollary. 
\begin{corollary}\label{cor:recur}
	Consider a $(s,\boldsymbol{\xi})-\Name$ process with $s$ even and environment $\boldsymbol{\xi}$ which satisfies condition  \eqref{cond:S}. Then the walker changes its parity infinitely often almost surely.
	\end{corollary}




\subsection{Null recurrence driven by distinct local behaviors in i.i.d environment} \label{sec:iid}
In this part, we consider the particular case where the environment process $\boldsymbol{\xi}$ is a sequence of i.i.d. random variables. The goal is to extract finer information about the exiting time $\tau_{\mathrm{exit}}$, specifically, regarding its expected value from different initial conditions.

By our discussion about conditions \eqref{cond:S} and \eqref{cond:I} at the beginning of this section and Proposition~\ref{prop:null-recurr} it follows that under i.i.d. environments with non-zero finite mean the~$(2k,\boldsymbol{\xi})-\Name$ is null recurrent. However, we can identify two distinct regimes of the null recurrence for these processes. 
These two regimes are explained in terms of the random walk $Z$ defined { in  Equation~\eqref{def:Z}}. Roughly speaking what we observe is that, under certain environment conditions, the random walk $Z$ evolves in its graph with infinite mean time transition and under other environment conditions, it indeed evolves with finite mean time transitions.

Below we state the main result of this section and before proving it we explain how it is related with the two possible behaviors of the auxiliary random walk $Z$. Moreover, in order to simplify the statement of the theorem, we use the following notation
\[
\mathbb{P}_{\ell+1,\Root, 2k, \boldsymbol{\xi}} \text{ and } \mathbb{E}_{\ell+1,\Root, 2k, \boldsymbol{\xi}}\;,
\]
to denote the \Name{} process whose initial state $(T_0,x_0)$ is the root with $\ell$ leaves and a self-loop.
\begin{theorem}\label{thm:totti} Consider a $(2k, \boldsymbol{\xi})-\Name$ process, with  
independent environment $\boldsymbol{\xi}=\{\xi_{j}\}_{j \in \N}$  satisfying condition  \eqref{cond:Mr}$_1$.
Then, $\Name$ is null recurrent. Furthermore:
	\begin{enumerate}
		\item[(i)] If $k < \mu := E(\xi_1)$, then $\Ed_{\ell +1, \Root, 2k, \boldsymbol{\xi}} \left(\tau_{\rm exit}\right) = \infty$ for every $\ell \ge 1$; \\
		
		\item[(i*)] If $\xi_1 \equiv k$ or $k = \mu$ and $\xi_1$ satisfies
		\begin{equation} \label{modd}
		\limsup_{n \rightarrow \infty} \frac{\log(n)^2}{n} \log P\left(\xi > n/\log(n)\right) = -\infty\;,
		\end{equation}
		then $\Ed_{\ell +1, \Root, 2k, \boldsymbol{\xi}} \left(\tau_{\rm exit}\right) = \infty$ for every $\ell \ge 1$; \\
		
		\item[(ii)] If $k > \mu$ then there exists $\gamma > 0$ such that $\Ed_{\ell +1, \Root, 2k, \boldsymbol{\xi}} \left(\tau_{\rm exit}\right) \ge \gamma \ell$ for every $\ell \ge 1$. Moreover, if there exists $\varepsilon>0$ such that $\boldsymbol{\xi}$ satisfies condition \eqref{cond:Mr}$_{2+\varepsilon}$, then 
		\[
		\Ed_{\ell +1, \Root, 2k, \boldsymbol{\xi}} \left(\tau_{\rm exit}\right) = \Theta(\ell)\;.
		\]
	\end{enumerate}
\end{theorem}
Let us say some words about the above results. The reader may notice the lack of a general initial state on the statement above, however, the general case is implicitly covered. Note that regardless of the initial state of the $\Name$ process items \textit{(i)} and \textit{(i*)} imply that when $X$ lands on a leaf at an even step by the Strong Markov Property the {process $Z$} will take, in average, infinity time to move, although it will move almost surely since $X$ is recurrent. On the other hand, as we explained before the more non-leaves neighbors a vertex has the more likely is $X$ to leave it, thus in item \textit{(ii)}, under higher moment conditions, $X$ always moves in finite mean time.

We will prove items \textit{(i)} and \textit{(ii)} separately.  The proof of item \textit{(i*)} is similar to that of item \textit{(i)}, but it requires a moderate deviations result for sums of i.i.d random variables. Condition (\ref{modd}) is a condition to guarantee this moderate deviations result. This sort of condition implies finite second moment and is implied by the existence of exponential moments. The interested reader may consult \cite{EL}.


\begin{proof}[Proof of (i) and (i*) in Theorem \ref{thm:totti}] We begin noticing that 
	\[
	\mathbb{E}_{\ell +1, \Root, 2k, \boldsymbol{\xi}}\left(\tau_{\rm exit}\right) <\infty \iff \sum_{n\geq 1} \Pd_{\ell +1, \Root, 2k, \boldsymbol{\xi}} \left(\tau_{\rm exit}>2kn\right)<\infty\;.
	\] 
Since in this case the initial state will be fixed, we omit all the indices of $\Pd$ and $\Ed$ throughout the proof and let $d$ be $\ell +1$. 
From \eqref{tailtau} we have
\begin{align}\label{eq:1}
\sum_{n\geq 1}\mathbb{P}\left(\tau_{\rm exit}>2kn\right) = E \left(\sum_{n\geq 1}\exp\left\{
k\sum_{i=0}^{n-1}\log\left(1-\frac{1}{d+S_i}\right)\right\} \right)\;.
\end{align}
For $n_0\geq 1$ a finite positive integer (to be suitably chosen later)  we can decompose the  summation in Equation~\eqref{eq:1} as 
\begin{align*}
E \left(\sum_{n=1}^{n_0}\exp\left\{
k\sum_{i=0}^{n-1}\log\left(1-\frac{1}{d+S_i}\right)\right\} \right) + E \left(\sum_{n\geq n_0+1}\exp\left\{
k\sum_{i=0}^{n-1}\log\left(1-\frac{1}{d+S_i}\right)\right\} \right)\;.
\end{align*}
Since the first term satisfies
\[
E \left(\sum_{n=1}^{n_0}\exp\left\{
k\sum_{i=0}^{n-1}\log\left(1-\frac{1}{d+S_i}\right)\right\} \right)\leq n_0<\infty\;,
\]
we obtain that 
\[
\mathbb{E}\left(\tau_{\rm exit}\right)<\infty \iff E \left(\sum_{n\geq n_0+1}\exp\left\{
k\sum_{i=0}^{n-1}\log\left(1-\frac{1}{d+S_i}\right)\right\} \right):=\beta < \infty \;.
\]
The above defined quantity $\beta$ satisfies the following bound:
\begin{align*}
\quad \beta &\geq \exp\left\{kn_0\log\left(1-\frac{1}{d}\right)\right\}E \left(\sum_{n\geq n_0+1} \exp\left\{
k\sum_{i=n_0}^{n-1}\log\left(1-\frac{1}{d+S_i}\right)\right\} \right)\;,
\end{align*}
where we used that  $e^{\sum_{i=1}^{n_0-1}\log \left(1-\frac{1}{d+S_i} \right)}\leq 1$ pointwise. 
\smallskip

Now before we continue the proof, we will need specify some variables that will depend on the cases presented in $(i)$ and $(i*)$. 

\medskip

\noindent \underline{Case $k < \mu$:} By the Strong Law of Large Numbers, $\frac{S_n}{n}$ converges to $\mu$ almost surely.
We shall need something stronger and in particular that the convergence above is uniform in some subset of positive probability. Using Egorov's theorem, we know that for all $\varepsilon>0$ there exists a measurable set $B_\varepsilon$ with $P(B_\varepsilon)<\varepsilon$, such that $\frac{S_n}{n}$ converges uniformly to $\mu$ on $B^c_\varepsilon$. Due to the uniform convergence,  $\forall \omega \in B^c_\varepsilon$, $\forall \delta >0$ there exists~$n_0=n_0(\delta)$ such that for all~$n\geq n_0$ it holds that 
\[
n(\mu - \delta)\leq S_n(\omega) \leq n(\mu+\delta)\; .
\]
Here we fix $\delta_n = \delta$ for every $n \ge 1$. \\

\noindent \underline{Case $\xi \equiv k$:} This case is the simplest one, but the reader can follow the rest of the proof considering $B_\varepsilon = \emptyset$ for all $\varepsilon >0$ and $\delta_n = 0$ for every $n \ge 1$. \\

{
\noindent \underline{Case $k = \mu$ under Equation (\ref{modd}):} Here we rely on Theorem 2.2 in \cite{EL}. The result is the following: Let $(\zeta_n)_{n\ge 1}$ be a sequence of i.i.d. integrable random varibles with mean zero and $(b_n)_{n\ge 1}$ be a sequence of real numbers such that
\begin{equation}
\label{MDP}
\frac{b_n}{\sqrt{n}} \uparrow \infty \quad and \quad \frac{b_n}{n} \downarrow 0\;.
\end{equation}
Then 
$$
\lim_{n \rightarrow \infty} \frac{n}{b_n^2} \log \big( n P(|\zeta_1|> b_n) \big) = - \infty\;,
$$
is equivalente to the claim that $\frac{1}{b_n} \sum_{j=1}^n \zeta_j$ satisfies a moderate deviation principle with good rate function. This moderate deviation principle has speed $b_n^2/n$ (see the end of page 212 in \cite{EL}). We will not state the moderate deviation principle here, we just point out that it implies that for every $M > 0$, there exists $c = c(M) > 0$ such that
$$
P\left( \frac{1}{b_n} \sum_{j=1}^n \zeta_j > M \right) \le \exp\left \lbrace-c \frac{b_n^2}{n}\right \rbrace\;,
$$
for all $n\ge 1$. Here $\zeta_n = \xi_n - E[\xi_n]$ and we choose $b_n = 1$ and $b_n = n/\log(n)$ for $n\ge 2$, the reader can check that \eqref{modd} is equivalent to \eqref{MDP} (since $\xi$'s are non-negative, we don't need to subtract the mean value in \eqref{modd}). Thus a moderate deviation result holds and there exists a constant $c > 0$ such that
$$
P\left( \left| \frac{S_n}{n} - \mu \right| > \frac{k}{2\log(n)} \right) \le \exp\left \lbrace-c \frac{n}{\log(n)^2}\right \rbrace \, ,
$$
for all $n\geq 1$.
By Borel-Cantelli Lemma, we have that almost surely there exists $\tilde{n}_0$ (random) such that
$$
n \left( \mu - \frac{k}{2\log(n)} \right) \leq S_n \leq n \left( \mu + \frac{k}{2\log(n)} \right) \ \ \forall \, n\ge \tilde{n}_0\;.
$$ 
Then, simple arguments allow us to obtain uniform bounds on large sets, i.e, for all $\varepsilon > 0$ there exists a measurable set $B_\varepsilon$ with 
$P (B_\varepsilon)<\varepsilon$ and $n_0 > 0$, such that $\forall \, \omega \in B^c_\varepsilon$ and $\forall \, n\geq n_0$ it holds that 
\[
n(\mu - \delta_n)\leq S_n(\omega) \leq n(\mu+\delta_n)\; ,
\]
where $\delta_n = k/2\log(n)$.}

Consider the sets $B_\varepsilon$, the integers $n_0$ and the sequence $(\delta_n)_{n\ge 1}$ defined as above according to the distinct cases. By simply integrating over $B_{\varepsilon}^c$ we have
\begin{equation*}
E \left(\sum_{n\geq n_0+1} \exp\left\{k\sum_{i=n_0}^{n-1}\log\left(1-\frac{1}{d+S_i}\right)\right\} \right) \geq 
E \left(\sum_{n\geq n_0+1} \exp\left\{
k\sum_{i=n_0}^{n-1}\log\left(1-\frac{1}{d+S_i}\right)\right\}\mathbb{1}_{B^c_{\varepsilon}} \right).
\end{equation*}
Using the above bound and recalling that $S_n(\omega)\geq n(\mu-\delta_n)$ for every $n \ge n_0$ on $B^c_{\varepsilon}$ we obtain
\begin{align*}
\beta &\geq  \exp\left\{kn_0\log\left(1-\frac{1}{d}\right)\right\}E \left(\sum_{n\geq n_0+1} \exp\left\{
k\sum_{i=n_0}^{n-1}\log\left(1-\frac{1}{d+S_i}\right)\right\} \mathbb{1}_{B^c_{\varepsilon}}\right)
\\&\geq  \exp\left\{kn_0\log\left(1-\frac{1}{d}\right)\right\}E
\left(\sum_{n\geq n_0+1} \exp\left\{
k\sum_{i=n_0}^{n-1}\log\left(1-\frac{1}{d+i\left(\mu-\delta_i \right)}\right)\right\}\mathbb{1}_{B^c_{\varepsilon}} \right)\;.
\end{align*}
Since $\log(1-x)\geq -x -\frac{x^2}{1-x}$,  for $0\leq x<1$, we have that  
\begin{align*}
\sum_{i=n_0}^{n-1}&\log\left(1-\frac{1}{d+i\left(\mu-\delta_i \right)}\right)\geq \\&-\sum_{i=n_0}^{n-1}\frac{1}{d+i\left(\mu -\delta_i \right)} - \sum_{i=n_0}^{n-1}\frac{1}{\left(d+i \left(\mu -\delta_i \right)\right)\left(d+i\left(\mu -\delta_i \right)-1\right)}\;.
\end{align*}
Note that the second term on the RHS of the above inequality is finite. I.e.,
\[
L:=\sum_{i=1}^{\infty}\frac{1}{\left(d+i\left(\mu -\delta_i \right)\right)\left(d+i\left(\mu - \delta_i \right)-1\right)}< + \infty\;,
\]
hence 
\begin{align*}
\beta &\geq \exp\left\{kn_0\log\left(1-\frac{1}{d}\right)\right\}e^{-kL}P(B^c_{\varepsilon})\sum_{n\geq n_0+1} \exp\left\{
-k\sum_{i=n_0}^{n-1}\frac{1}{d+i\left(\mu -\delta_i \right)}\right\}
\\
&= c \sum_{n\geq n_0+1} \exp\left\{
-k\sum_{i=n_0}^{n-1}\frac{1}{d+i\left(\mu - \delta_i \right)}\right\}\;,
\end{align*}
where $c= e^{kn_0\log\left(1-\frac{1}{d}\right)}e^{-kL}P(B^c_{\varepsilon})$.
Thus, we are left with showing that the above  summation diverges. 
Recall that for every decreasing function $f$ it holds that 
\[
\int_{a}^{b+1} f(t)dt \leq \sum_{i=a}^b f(i)\leq \int_{a-1}^{b} f(t)dt\;.
\]
For the case $\delta_n = \delta$ for every $n$ (even when $\delta = 0$), we have that
\begin{align*}
&-\sum_{i=n_0}^{n-1}\frac{1}{d+i\left(\mu-\delta\right)}\geq -\int_{n_0-1}^{n-1}\frac{1}{d+t\left(\mu-\delta\right)}dt\\
&= -\frac{1}{\mu-\delta}
\log\left(d+(n-1)(\mu-\delta_i)\right)+ \frac{1}{\mu-\delta}
\log\left(d+(n_0-1)(\mu-\delta)\right)\;.
\end{align*}
Thus,
\begin{align*} 
\beta &\geq c \;e^{\frac{k}{\mu-\delta}\log\left(d+(n_0-1)(\mu-\delta)\right)}\sum_{n\geq n_0+1}e^{-\frac{k}{\mu-\delta} \log\left(d+(n-1)(\mu-\delta)\right)}
\\ \nonumber
&=C\sum_{n\geq n_0+1} \frac{1}{\left( d+(n-1)(\mu-\delta)\right)^\frac{k}{\mu-\delta} }\;.
\end{align*}
By the assumption that $k<\mu$, we know that there exists a $\varepsilon_0>0$ such that $k=\mu-\varepsilon_0$. Choosing $\delta=\varepsilon_0$ and $n_0$ accordingly we obtain that the right-hand side diverges and so thus~$\beta$, proving the items $(i)$ and $(i*)$ for the particular case in which $\xi$ is constant. 

In order to conclude the proof of these items, now consider the remainder cases $k = \mu$ and $\delta_n = k/2\log(n)$. Then  
\begin{align*}
-k\sum_{i=n_0}^{n-1}\frac{1}{d+i\left(\mu-\delta_i \right)} & \geq 
 - \sum_{i=n_0}^{n-1}\frac{1}{i(1-1/2\log(i))} \\
& \geq - \sum_{i=n_0}^{n-1}\frac{1}{i}  - \sum_{i=n_0}^{n-1}\frac{1}{i\log(i)}  \\ 
& \geq \gamma_0 - \log(n) - \log\log(n) \; ,
\end{align*}
where $\gamma_0$ is a positive constant that only depends on $n_0$ (here the reader can use the integration step as above having in mind that the primitive of $1/t\log(t)$ is $\log\log(t)$. Therefore
$$
\beta \ge C \sum_{n\geq n_0+1} \frac{1}{n \log(n)} = \infty \; .
$$
\end{proof}
We now prove the second and last part of Theorem~\ref{thm:totti}. 
\begin{proof}[Proof of (ii) in Theorem \ref{thm:totti}] We begin with proving a lower bound for the expectation of the exit time. 

\noindent \underline{\textit{Step one:} lower bound on $\Ed_{\ell +1, \Root, 2k, \boldsymbol{\xi}} \left(\tau_{\rm exit}\right)$.} Recall that
\begin{equation*}
\mathbb{E}\left(\tau_{\rm exit}\right)\geq \sum_{n\geq 1}\mathbb{P}\left(\tau_{\rm exit}\geq sn\right)=\sum_{n\geq 1}E\left(\exp\left\{k\sum_{i=0}^{n-1}\log(1-\frac{1}{d+S_i})\right\}\right)\;.
\end{equation*}
Since $\log (1-x)\geq -x-x^2\,$, for $0\le x \le \frac{1}{2}$ we get
\begin{align*}
\mathbb{E}\left(\tau_{\rm exit}\right) \ge \sum_{n\ge 1}E\left( \exp\left\{-k\left[\sum_{i=0}^{n-1}\frac{1}{d+S_i} + \sum_{i=0}^{n-1}\frac{1}{(d+S_i)^2}\right ]\right\} \right)\;.
\end{align*}
By Large Law of Large Numbers $\frac{S_i}{i} - \mu \to 0$ as $i\to\infty$ (a.s.). Hence given $0<\delta<\mu$ there exist (a.s.) some random $M$ such that for all $i\ge M$, we have $S_i\geq (\mu-\delta)i$. Then for $n > M$,
\begin{align*}
&\sum_{i=0}^{n-1}\frac{1}{d+S_i} + \sum_{i=0}^{n-1}\frac{1}{(d+S_i)^2}\\
&= \sum_{i=0}^{M}\frac{1}{d+S_i} + \sum_{i=0}^{M}\frac{1}{(d+S_i)^2} + \sum_{i=M+1}^{n-1}\frac{1}{d+S_i} + \sum_{i=M+1}^{n-1}\frac{1}{(d+S_i)^2}\\
&\le \sum_{i=0}^{M}\frac{1}{d+S_i} + \sum_{i=0}^{M}\frac{1}{(d+S_i)^2} +  \sum_{i=M+1}^{n-1}\frac{1}{d+i[\mu-\delta]} + \sum_{i=M+1}^{n-1}\frac{1}{(d+i[\mu-\delta])^2}\\
&\le 2(M+1) + \frac{1}{\mu-\delta}\left[\log\left(d+[\mu-\delta](n-1)\right)-\log\left(d+[\mu-\delta]M\right)\right] +\varepsilon(d)\;,
\end{align*}
where $\varepsilon(d)\to 0$ as $d\to\infty$. By Egorov's theorem, we can find some measurable set $B$ and some constant $M$ such that  $P(B)>0$ and on $B$, $S_i\geq [\mu-\delta]i$, for all $i\geq M$. {Hence,
\begin{align*}
\mathbb{E}\left(\tau_{\rm exit}\right)&\ge\sum_{n\geq 1}\int_{B}\exp\left\{-k \left[2(M+1) + \frac{\left[\log\left(d+[\mu-\delta](n-1)\right)-\log\left(d+[\mu-\delta]M\right)\right]}{\mu-\delta} +\varepsilon(d)    \right]\right\}dP\\
&=\exp\left\{-k \left[2(M+1)+\varepsilon(d)\right]\right\} \left(d+[\mu-\delta]M\right)^{\frac{k}{\mu-\delta}} \sum_{n\geq 1}\frac{1}{\left(d+[\mu-\delta](n-1)\right)^{\frac{k}{\mu-\delta}}}P(B)\;.
\end{align*}
	To simplify notation let us denote by $\beta := \mu - \delta$ and $\alpha := k/\beta$. Then,
	\begin{equation*}
		\sum_{n\ge 1}\frac{1}{\big[d + (\mu-\delta)(n-1)\big]^{\frac{k}{\mu - \delta}}} \ge \int_0^{+\infty}\frac{1}{(d+\beta x)^{\alpha}}dx = \frac{1}{\beta(\alpha -1 )d^{\alpha - 1}}\; .
	\end{equation*}
	Hence,
	\begin{align*}
		&\big[d+(\mu - \delta)M\big]^{\frac{k}{\mu - \delta}}\sum_{n\ge 1}\frac{1}{\big[d + (\mu-\delta)(n-1)\big]^{\frac{k}{\mu - \delta}}}\\&\ge (d + \beta M)^{\alpha}\frac{1}{\beta(\alpha -1 )d^{\alpha - 1}}\\
		&=\frac{d}{\beta(\alpha - 1)}\Big(1+\frac{\beta M}{d}\Big)^{\alpha}\; .
	\end{align*}
	Thus,
	\begin{equation*}
		\mathbb{E}\left(\tau_{\rm exit}\right)\ge \frac{d}{\beta(\alpha - 1)}\exp\big\{-k [2(M+1) + \epsilon(d)]\big\}\left(1+\frac{\beta M}{d}\right)^{\alpha},
	\end{equation*}
	where,
	\begin{equation*}
		\lim_{d\to\infty}\exp\big\{-k[2(M+1) + \epsilon(d)]\big\}\left(1+\frac{\beta M}{d}\right)^{\alpha} = \exp\big\{-2k(M+1)\big\}\; .
	\end{equation*}
	Thus we can find some positive $C$, which does not depend on d, such that
	\begin{equation*}
		\mathbb{E}\left(\tau_{\rm exit}\right)\ge Cd\; .
	\end{equation*} }
This proves that
\[
\Ed_{\ell +1, \Root, 2k, \boldsymbol{\xi}} \left(\tau_{\rm exit}\right) \ge C(\ell +1)\;.
\]
\noindent \underline{\textit{Step two:} upper bound on $\Ed_{\ell +1, \Root, 2k, \boldsymbol{\xi}} \left(\tau_{\rm exit}\right)$.} In this part we are going to prove the following upper bound
\begin{equation*} 
\mathbb{E}\left(\tau_{\rm exit}\right) \le C_1 + C_2 \ell^{1+\gamma} \sum_{n\geq 1} \frac{1}{\left(\ell + 1 + n \left(\mu+\delta\right)\right)^{1 + \gamma} }\;,
\end{equation*}
for environments under the assumption (\ref{cond:Mr})$_{2+\varepsilon}$. {Then an upper bound of order $\Theta(\ell)$ is obtained from estimates on the above series 
as in step one.}

For fixed $\delta>0$, let us define $M_{\delta}$ the following random variable
\begin{equation*}
M_\delta(\omega):=\sup \left\{i\geq 1: \left\vert \frac{S_i(\omega)}{i} - \mu \right\vert\geq \delta\right\}\;.
\end{equation*}
Then, we may compute the $r$-th moment of $M_{\delta}$ as follows
\[
E \left(M_\delta^r\right)=\sum_{n\geq 1}n^{r-1}P\left(M_\delta\geq n\right) \leq \sum_{n\geq 1}n^{r-1}P \left( \sup_{j\geq n} \left\vert \frac{S_j(\omega)}{j} - \mu \right\vert\geq \delta \right)\;,
\]
and, by Theorem 3 in \cite{BK}, we have the characterization below
\begin{align}\label{eq:2}
E (\xi^{r+1})<\infty \iff \forall \delta>0, \sum_{n\geq 1}n^{r-1}P\left( \sup_{j\geq n} \left\vert \frac{S_j(\omega)}{j} - \mu \right\vert\geq \delta \right)<\infty\;. 
\end{align}
By Equation~\eqref{eq:1} we can write
\begin{align*}
&\mathbb{E}\left(\tau_{\rm exit}\right)= \\
&E \left(\sum_{n=1}^{M_\delta+1}\exp\left\{
k\sum_{i=0}^{n-1}\log\left(1-\frac{1}{d+S_i}\right)\right\} \right) + E \left(\sum_{n\geq M_\delta+2}\exp\left\{
k\sum_{i=0}^{n-1}\log\left(1-\frac{1}{d+S_i}\right)\right\} \right).
\end{align*}
For the first term in the RHS, we have that, for every $\delta>0$
\[
E \left(\sum_{n=1}^{M_\delta+1}\exp\left\{
k\sum_{i=0}^{n-1}\log\left(1-\frac{1}{d+S_i}\right)\right\} \right)\leq E (M_\delta+1)< \infty\;,
\]
due to Equation~\eqref{eq:2} and the hypothesis that $E(\xi^2)<\infty$, which implies that $E(M_\delta)<\infty$. Thus, we are left to bound from above the following term
\begin{align*}
&E \left(\sum_{n\geq M_\delta+2}\exp\left\{
k\sum_{i=0}^{n-1}\log\left(1-\frac{1}{d+S_i}\right)\right\} \right)
\\
&
\leq 
E \left(\sum_{n\geq M_\delta+2}\exp\left\{
k\sum_{i=M_\delta +1 }^{n-1}\log\left(1-\frac{1}{d+S_i}\right)\right\} \right)
\\
&
\leq 
E \left(\sum_{n\geq M_\delta+2}\exp\left\{
k\sum_{i=M_\delta+1}^{n-1}\log\left(1-\frac{1}{d+ i \left(\mu+\delta\right)}\right)\right\} \right)\;.
\end{align*}
Since $\log(1-x)\leq -x$,  for $0\leq x<1$, we have that  
\begin{align*}
\sum_{i=M_\delta +1}^{n-1}&\log\left(1-\frac{1}{d+i\left(\mu+\delta\right)}\right)\leq -\sum_{i=M_\delta+1}^{n-1}\frac{1}{d+i\left(\mu+\delta\right)}\;.
\end{align*}
Bounding the sum by the integral,
\begin{equation*}
\begin{split}
-k\sum_{i=M_\delta+1}^{n-1}\frac{1}{d+i\left(\mu+\delta\right)} & \leq -\int_{M_\delta+1}^{n}\frac{1}{d+t\left(\mu +\delta\right)}dt\\
& = -\frac{1}{\mu +\delta}
\log\left(d+n(\mu +\delta)\right)+ \frac{1}{\mu +\delta}
\log\left(d+(M_\delta+1)(\mu +\delta)\right)\;.
\end{split}
\end{equation*}
Thus, 
\begin{align*}
&
E \left(\sum_{n\geq M_\delta+2}\exp\left\{
k\sum_{i=M_\delta+1}^{n-1}\log\left(1-\frac{1}{d+ i \left(\mu+\delta\right)}\right)\right\} \right)
\\
&\leq 
E \left(\left((d+M_\delta +1)(E(\xi)+\delta)\right)^{\frac{k}{E(\xi)+\delta }} \sum_{n\geq M_\delta+2} \frac{1}{\left(d+ n \left(\mu +\delta\right)\right)^{\frac{k}{\mu +\delta}}} \right)
\\
&\leq 
E \left(\left((d+M_\delta +1)(\mu +\delta)\right)^{\frac{k}{\mu +\delta }}\right) \sum_{n\geq 1} \frac{1}{\left(d+ n \left(\mu +\delta\right)\right)^{\frac{k}{\mu +\delta}}} \;.
\end{align*}

Using the hypothesis that $k>\mu$, we know that there exists a $\varepsilon_0>0$ such that $k=\mu+\varepsilon_0$. Then, for every  $\delta < \varepsilon_0$ we obtain that
\[
\sum_{n\geq 1} \frac{1}{\left(d+ n \left(\mu+\delta\right)\right)^{\frac{k}{\mu+\delta}}}<\infty\;.
\]
On the other hand, using the hypotheses that there exists a $\varepsilon>0$ such that 
$E\left(\xi^{2+\varepsilon}\right)<\infty$ 
and knowing that $E\left(\xi^{r+1}\right)<\infty \implies E \left(M_\delta^{r}\right)<\infty$, we have that for every  $\delta > \frac{\varepsilon_0 -\varepsilon \mu}{1+\varepsilon}$ ,
\[
E \left( M_\delta^{\frac{k}{\mu +\delta}} \right)<\infty\;. 
\]
Choosing $\delta$ such that $\frac{\varepsilon_0 -\varepsilon \mu}{1+\varepsilon}< \delta<\varepsilon_0$, we conclude the proof of the second step. 

\end{proof}

\begin{remark}
It is interesting to note that the local behavior described in Theorem \ref{thm:totti} allows us to prove null recurrence, already established in Proposition \ref{prop:null-recurr} under much more general conditions on the environment. If $s/2 < \mu$, or $s/2 = \mu$ and \eqref{modd} holds, then transitions between vertices of the same parity take infinite mean times, thus null-recurrence follows immediately. For $s/2 > \mu$ we just pass the idea of how we can use that $\Ed_{\ell +1, \Root, 2k, \boldsymbol{\xi}} \left(\tau_{\rm exit}\right) = \Theta(\ell)$ to show null recurrence. For the sake of simplicity we consider $s=2$ which implies that $\mu < 1$. Observe that Theorem \ref{thm:totti} can be applied analogously to the loop process and, by using our coupling, the proof of null recurrence for the loop process follows the same lines as the proof for the $\Name{}$. {Recall also that every use of the self-loop is equivalent to two steps of the $\Name{}$. 
So consider a loop process $Y$ in $\{0,...,k\}$, $k\ge 2$, so that on the $j$-th time spent at vertex $1$, $\xi_j$ new loops are added to that vertex, where $\{\xi_j\}_{j\ge 1}$ is a sequence of i.i.d. integer valued random variables with $E\left(\xi_j\right) = \mu < 1$. Suppose also that $Y_0 = 1$ and vertex $1$ has initially $L-2$ loops attached to it. Now let $N$ be the number of times $Y$ visits $1$ before it visits $0$ for the first time such that uses of the self-loop do not count as a new visit. Formally
$$
N = 1 + \# \{ n : Y_{n-1} = 2, Y_{n}=1, Y_j \neq 0, j=1,...,n-2 \}\;. 
$$}
Then $N$ is stochastically bounded from below by a geometric random variable of parameter $L^{-1}$ (new loops may be attached to 1). Therefore $\Pd \left( N = r \right) \ge e^{-c \frac{r}{L}}$. Define also $\tau^i_{\rm exit}$ as the number of times that $Y$ uses a self-loop upon the $i$-th visit to $1$. Take $\gamma$ as in (ii) in the statement of Theorem \ref{thm:totti}. Note that
$$
\mathbb{E} \left(\tau_{\rm exit}^1\right) \ge \gamma \ge \gamma \mu \; .
$$
{
We will show by induction that
\begin{equation}
\label{estnull}
1 + \sum_{i=1}^r \mathbb{E} \left(\tau_{\rm exit}^i\right) \ge (1 + \gamma \mu)^{r} \; .
\end{equation}
Suppose that \eqref{estnull} holds for some $r\ge 1$ and let $l_i$ be the number of self-loops added to site $1$ on the i-th visit, then
each $l_i$ is the sum of $\tau_{\rm exit}^{i}$ distinct variables taken from the sequence $\{\xi_j\}_{j\ge 1}$ and one can check that 
$\mathbb{E}\left(l_i\right) = \mu \mathbb{E} \left(\tau_{\rm exit}^{i}\right)$. Therefore, by (ii) in Theorem \ref{thm:totti}, we have that
$$
\mathbb{E} \left(\tau_{\rm exit}^{r+1}\right) = \mathbb{E}\left( \mathbb{E}\left(\tau_{\rm exit}^{r+1} | l_1,...,l_r\right)  \right) 
\ge \mathbb{E} \big[ \gamma \big( 1+\sum_{i=1}^r l_i \big) \big] \ge \gamma \mu \big( 1 + \sum_{i=1}^r \mathbb{E} \left(\tau_{\rm exit}^{i}\right) \big) 
= \gamma \mu (1 + \gamma \mu)^{r} \;, 
$$
and
$$
1 + \sum_{i=1}^{r+1} \mathbb{E} \left(\tau_{\rm exit}^i\right) \ge  (1 + \gamma \mu)^{r} + \gamma \mu (1 + \gamma \mu)^{r} = (1 + \gamma \mu)^{r+1}\;.
$$
Thus we have proved \eqref{estnull}.} Now put $\nu$ as the first hitting time of $0$ by Y, then 
$$
\mathbb{E} \left( \nu | N = r\right) \ge 1+ \sum_{i=1}^{r} \mathbb{E} \left(\tau_{\rm exit}^i\right) \ge (1 + \gamma \mu)^{r}  \, ,
$$
for every $r \ge 1$. Thus
$$
\mathbb{E} \left( \nu \right) = \sum_{r=1}^{\infty} \mathbb{E} \left( \nu | N = r\right) \, \Pd \left( N = r \right) \ge \sum_{r=1}^{\infty} \Big( (1 + \gamma \mu) e^{-\frac{c}{L}} \Big)^r \, .
$$
Now simply choose $L$ such that $(1 + \gamma \mu) e^{-\frac{c}{L}} > 1$, the above series will diverge and $\mathbb{E} \left( \nu \right) = \infty$.
\end{remark}

%% file: ballisticity-s-odd-Rodrigo.tex

\section{The Case s odd: Ballisticity of the walker in {\Name{}}}\label{sec:odd}
In this section we prove Theorem \ref{theo:ballistic}, which states ballistic behavior of $\Name{}$ for $s$ odd  under condition (\ref{cond:UE}). To show ballisticity, i.e.,  
\begin{equation*}
\liminf_{n \rightarrow \infty}\frac{\mathrm{dist}_{T_n}(X_n,\Root)}{n}  \geq c, \; \Pd_{T_0,x_0,s,\bxi}\text{- almost surely}\;,
\end{equation*}
 we rely on a general criterion which we extrapolate from the proof of ballisticity for BGRW given in \cite{FIORR}. We believe this general criterion may be of independent interest and may be applied to a wider class of similar processes.

\subsection{General criterion for ballisticity}
To provide some intuition about the general criterion for ballisticity in \Name{}, let us begin  
building a bridge with the classical theory of  Random Walk on Random Environments. In the latter context, the main mechanism behind ballisticity is the concept of \textit{regeneration time}, {see~\cite{S}}. Roughly speaking, it says the walk has regenerated if it does not return to a half-space after a certain time. This implies that from time to time, the walk is always exploring independent portions of the environment.
The idea of regeneration also appears in the context of the \Name{} when $s$ is odd. In essence, the regeneration is now due to two reasons combined:
\begin{itemize}
	\item[\textbf{(1)}] The walk is capable of building long enough paths regardless the  current tree structure;
	\item[\textbf{(2)}] Once the walk is at a tip of a path, it takes very long time to backtracking.  
\end{itemize}
{The two  items above assure that the walker may ``forget'' the tree structure built up to a certain time  and start afresh. }
Item (2) is related to the hitting times estimates proved in  Section~\ref{sec:longtimes}. 
Thus, item (2) holds in the presence of condition (\ref{cond:UE}) regardless the parity of~$s$.
However, as will be clearer in the sequel, item (1) requires $s$ to be odd. With $s$ odd, the walk may ``push the tree forward" adding new leaves to the bottom of the tree. This feature gives the walk the ability of creating the escape routes it needs. {This combined with (2) assures that the walker has a positive probability of never returning to some portions of the tree.}

The general criterion for ballisticity  introduced in this section is, in essence, a quantitative version of (1) and (2).

\medskip 
For $r$ a positive integer, let us denote by $\Omega_r$ the subset of $\Omega$ formed by all pairs $(T,x)$ such that $T$ has height at least $r$ and $\mathrm{dist}_T(x, \text{\Root}) \ge r$. Let $\eta_r$ be the first time $X$ hits the ancestor of its initial position at distance $r$ in the path connecting $X_0$ to the root of $T_0$,  i.e.,
\begin{equation*}
\eta_r := \inf \left \lbrace n \ge 0 \; \middle | \; \mathrm{dist}_{T_n} (X_n, \Root) = \mathrm{dist}_{T_0} (X_0, \Root) - r\right \rbrace\;.
\end{equation*}

We say that the process $\{(T_n,X_n)\}_{n \in \N}$ satisfies conditions (R) and (L) if: there exists~$\alpha \in (0,1)$, $\varepsilon\in\left(0,\frac{1}{2}\right)$ and $r$ such that 
\begin{equation}\tag{R}\label{eq:R}
\inf_{(T_0,x_0) \in \Omega} \Pd_{T_0,x_0,s,\boldsymbol{\xi}} \left( \exists m \le \exp\{r^{\alpha}\},\mathrm{dist}_{T_m} (X_m, \text{root}) \ge 2r \right) \geq 1 - \varepsilon/2\;,
\end{equation}
and
\begin{equation}\tag{L}\label{eq:L}
\sup_{(T_0,x_0) \in \Omega_r} \Pd_{T_0,x_0,s,\boldsymbol{\xi}} \left(\eta_r \le \exp\{r^{\alpha} \}\right) \leq  \frac{1}{2}-\varepsilon\;.
\end{equation}
Condition \eqref{eq:R} guarantees that in at most $\exp\{r^{\alpha}\}$ steps the walker will be at distance at least $2r$ from the root, with sufficiently high probability. To see why \eqref{eq:R} is related to (1), just consider a process started at the bottom of $T_0$. Then, the only way of increasing the distance from the root by $2r$ is actually building a path this long. 

Condition \eqref{eq:L} instead, is related to return times and assures that, regardless of the initial condition, the walker needs at least a stretched exponential time $\exp\{r^{\alpha}\}$  to climb a path of length $r$. As we shall see, a Markov chain $\{(T_n, X_n)\}_{n\in \N}$  satisfying both conditions, seen in a window of time of order at most $\exp\{r^{\alpha}\}$, is more likely to see the process $X$  increasing its distance from the root by $r$ units (\eqref{eq:R} for jump to the {\em right}) than decreasing it by the same amount (\eqref{eq:L} for jump to the {\em left}).

{In order to show ballisticity it will be convenient to observe the walker $X$ at certain stopping times. First fix a positive integer $r$. Then define
\begin{equation*}
\begin{split}
\sigma_0 & \equiv  0\;, \\
\sigma_1 &:= \exp\{r^{\alpha}\} \wedge \inf \big\lbrace n > 0 : \left(\mathrm{dist}_{T_{n}}(X_{n}, \text{root}) - \mathrm{dist}_{T_0}(X_0, \text{root})\right) \in \{-r,r\} \big \rbrace\;.
\end{split}
\end{equation*}
and by induction
\begin{equation*}
\begin{split}
\sigma_k &:= \sigma_{k-1} + \sigma_1 \circ \theta_{\sigma_{k-1}} ((X_n)_{n\ge 0})\;,
\end{split}
\end{equation*}
for $k\ge 2$, where $\theta_k$ is the forward time shift by $k$ units, i.e. $\theta_{k} ((X_n)_{n\ge 0}) = (X_{n+k})_{n\ge 0}$.
Clearly, these stopping times depend on $r$, but we omit such a dependency to avoid clutter. 
From the definition it follows that, for all $k$, $\sigma^{}_k$ is bounded from above by $k\exp\{r^{\alpha}\}$. We also set $\Delta d_{k+1}= \mathrm{dist}_{T_{\sigma_{k+1}}}(X_{\sigma_{k+1}}, \Root) - \mathrm{dist}_{T_{\sigma_k}}(X_{\sigma_k},\Root)$ and denote by $\tilde{\mathcal{F}}_k$ the $\sigma$-field  generated by the
process $\{(T_n,X_n)\}_{n \in \N}$ up to time $\sigma_k$. }
%
\begin{lemma}\label{lema:couplingsk}
Let $\{(T_n,X_n)\}_{n \in \N}$ be a process satisfying conditions \eqref{eq:R} and \eqref{eq:L}, $T_0$ a rooted locally finite tree, $x_0$ a vertex of $T_0$.  There exists $\varepsilon>0$ and $r$ such that for all $k$
	\begin{equation}\label{eq:Biden}
	\inf_{T_0,x_0} \Pd_{T_0,x_0} \left( \Delta d_{k+1} = r \middle | \tilde{\mathcal{F}}_k\right) \geq \frac{1}{2}(1+\varepsilon)\;. 
	\end{equation}
Thus, if  $\left \lbrace S_k \right \rbrace_{k \ge 0}$ denotes  a $\frac{1}{2}(1+\varepsilon)$-right biased simple random walk on $\Z$,  the process $\{ \mathrm{dist}_{T_{\sigma_k}}(X_{\sigma_k}, \Root)/r\}_{k \ge 0}$ and $\left \lbrace S_k \right \rbrace_{k \ge 0}$ starting from $\lfloor \mathrm{dist}_{T_{0}}(x_0, \Root)/r \rfloor$ can be coupled in such way that
	\[
	\Pd \left( \mathrm{dist}_{T_{\sigma_k}}(X_{\sigma_k}, \Root) \ge rS_k, \forall k \right) = 1\;.
	\]
\end{lemma}
Once we have at our disposal \eqref{eq:L} and \eqref{eq:R}, the proof of the above lemma is in line with  Lemma 5.1 in~\cite{FIORR}. We give here just a sketch of the proof, which we hope may convince the reader familiar with the technical details behind the idea.
\begin{proof}[Proof sketch of Lemma \ref{lema:couplingsk}] 
%
To see why the definition of the {stopping times} imply Equation~\eqref{eq:Biden}, {note that $\sigma_{k+1}$ given $\sigma_k$ can be described through three complementary events:  
\begin{equation*}
\begin{split}
(a) & \ \sigma_{k+1} = \inf\{n > \sigma_{k} : \mathrm{dist}_{T_{n}}(X_{n}, \text{root}) - \mathrm{dist}_{T_{\sigma_k}}(X_{\sigma_k}, \text{root}) = r \} \le \exp\{r^{\alpha}\}; \\
(b) & \ \sigma_{k+1} = \inf\{n > \sigma_{k} : \mathrm{dist}_{T_{n}}(X_{n}, \text{root}) - \mathrm{dist}_{T_{\sigma_k}}(X_{\sigma_k}, \text{root}) = -r \} \le \exp\{r^{\alpha}\}; \\ 
(c) & \ |\sigma_{k+1} - \sigma_{k}| = \exp\{r^{\alpha}\} \textrm{ and } |\mathrm{dist}_{T_{n}}(X_{n}, \text{root}) - \mathrm{dist}_{T_{\sigma_k}}(X_{\sigma_k}, \text{root})| < r, \ \forall \sigma_{k}< n \le \sigma_{k+1}.\\
\end{split}
\end{equation*}
 Furthermore,  if (b) has occurred, then $X$ has climbed up $r$ levels of its tree spending less than $\exp\{r^{\alpha}\}$ steps. However, by \eqref{eq:L}  this happens with probability at most $1/2 -\varepsilon$. If instead (c) has occurred,  then X has walked for $\exp\{r^{\alpha}\}$ steps and has neither visited the ancestor $z$ of $x_0$ at distance $r$ nor has increased its distance from $z$ by $r$. It is possible to show that the probability of this event is the same of  observing a process 
$X$ on the subtree hung from $z$ (thus rooted at $z$) that in $\exp\{r^{\alpha}\}$ steps does not get at distance $2r$ from the root $z$. By condition \eqref{eq:R}  this probability is at most $\varepsilon/2$, proving the Lemma.}

Relying on Equation~\eqref{eq:Biden}, the coupling with a right biased random $\{S_k\}_{k \in \N}$ on~$\Z$ can be constructed as follows: 
	\begin{itemize}
		\item whenever 
		\[
		\mathrm{dist}_{T_{\sigma_{k+1}}}(X_{\sigma_{k+1}}, \Root) - \mathrm{dist}_{T_{\sigma_k}}(X_{\sigma_k}, \Root) < r\;,
		\]
		we let $S_{k+1}$ move to the left;
		
		\item otherwise, we decide according to another source of randomness independent of $\{(T_n,X_n)\}_{n \in \N}$, whether $S_{k+1}$ follow the process $\{\mathrm{dist}_{T_{\sigma_{k+1}}}(X_{\sigma_{k+1}}, \Root)/r\}_{k \in \N}$ or moves to the left.
	\end{itemize}
	Specifically, if the walker $X$ has not increased its distance by $r$ units taking less than $\exp\{r^{\alpha}\}$ steps, $S$ moves to the left. On the other hand, if $X$ has successfully increased its distance by $r$ units in the right amount of time, $S$ decides according to some coin whether it jumps to the left or to the right. The extra source of randomness is needed in order to make the increments of $\{S_k\}_{k \in \N}$ independent, although they depend on~$\{(T_n,X_n)\}_{n \in \N}$.
\end{proof}

As a consequence of the above coupling, we can now easily prove that a \Name{} satisfying \eqref{eq:L}  and \eqref{eq:R} is ballistic. 
\begin{proposition}[General criterion for Ballisticity]\label{prop:liminf} Let $\{(T_n,X_n)\}_{n \in \N}$ be a process satisfying conditions \eqref{eq:L}  and \eqref{eq:R}, then there exists a positive constant $c$, such that
	\[
	\liminf_{n \rightarrow \infty}\frac{\mathrm{dist}_{T_n}(X_n,\Root)}{n}  \geq c, \; \Pd_{T_0,x_0, s, \boldsymbol{\xi}}\text{- almost surely}\;,
	\]
	for all initial conditions $(T_0,x_0)$.
\end{proposition}
\begin{proof} By Strong Law of Large Numbers, given a $\frac{1}{2}(1+\varepsilon)$-biased random walk $\{S_k\}_{k\in \N}$ and Lemma~\ref{lema:couplingsk} we already have that for any initial condition $(T_0,x_0)$, 
	\begin{equation*}
	\liminf_{k \rightarrow \infty} \frac{\mathrm{dist}_{T_{\sigma_k}}(X_{\sigma_k}, \Root) }{k} \ge  \varepsilon \,r\;, \; \Pd_{T_0,x_0,s,\boldsymbol{\xi}}\text{- almost surely}\;.
	\end{equation*}
To pass from the subsequence to the whole sequence is a standard argument. {The key point is to observe that by definition of the stopping times, it follows that}
\begin{align*}
|\mathrm{dist}_{T_{\sigma_{k+1}}}(X_{\sigma_{k+1}}, \Root) - \mathrm{dist}_{T_{\sigma_k}}(X_{\sigma_k}, \Root)| \leq  r
\quad 
\text{ and } 
\quad
| \sigma_{k+1}- \sigma_{k}| \le \exp\{r^{\alpha}\}\;,
\end{align*}
hold almost surely, for every $k$. 
The reader may check the details in  Proposition 5.3 in \cite{FIORR}.
\end{proof}

\subsection{Proof of Theorem \ref{theo:ballistic}
}

In light of Proposition~\ref{prop:liminf}, in order to prove Theorem~\ref{theo:ballistic}, it is enough to show that the \Name{} process with $s$ odd and $\boldsymbol{\xi}$ is an independent environment satisfying condition (\ref{cond:UE}) fulfills conditions \eqref{eq:L} and \eqref{eq:R}. 
%

\medskip 

We begin recalling Corollary \ref{cor:etaybound}, which states that for $s$ odd and under condition~(\ref{cond:UE}) there exists a positive $C$ depending on $s$ and $\boldsymbol{\xi}$ only such that
	\[
\Pd_{T_0,x_0,s, \boldsymbol{\xi}} \left( \eta_{z} \le e^{\sqrt{\ell}}\right) \le \frac{C}{\sqrt{\ell}}\;.
\]
By setting $\ell=r^{2\alpha}$ with $\alpha \in (0,1)$ and choosing $r$ sufficiently large the above upper bound implies condition \eqref{eq:L}, since the above bound is uniform for $(T_0,x_0) \in \Omega_r$.

\medskip 
Thus, in order to prove Theorem \ref{theo:ballistic} we are left to prove that for $s$ odd and under (UE) the $\Name{}$ satisfies \eqref{eq:R}. However,  instead of showing it directly, we will actually show that an auxiliary condition (R)$_M$ is satisfied, which implies \eqref{eq:R}  for some values of $M$. 

Let us define the condition (R)$_M$: we say the \Name{} satisfies (R)$_M$ if there exists $n_0 = n_0 (s, M, \bxi) \in  \mathbb{N}$, depending only on $s, M$ and $\bxi$ such that, for all $n \geq  n_0$, all finite trees $T_0$ and all $x_0, y \in T_0$,
\begin{equation}\tag{R$_M$}
\inf_{(T_0,x_0) \in \Omega} \mathbb{P}_{T_0, x_0, s, \bxi}\left(\exists m \leq n : \dist{m}{X_{m}}{y}\geq \log^M n  \right)\geq 1-e^{-n^{1/4}}.
\end{equation}
Note that (R)$_M$, for $M>1$, implies \eqref{eq:R}: let $n= \exp\{r^{\alpha}\}$, choose $\alpha$ such that  $\alpha M >1$, and choose a large enough $r$. Thus, to  prove Theorem~\ref{theo:ballistic}, we are left with showing that \Name{} process satisfies condition 
(R)$_M$ for some $M>1$. 
For the sake of clarity and organization, we will divide the latter proof into  subsections, each one  corresponding to a step of the proof.
\subsubsection{General idea of the proof} The proof that \Name{} satisfies condition \eqref{eq:R} is similar to the proof for the BGRW (which is the \Name{} for $s=1$ and $\boldsymbol{\xi}$  an i.i.d. sequence of Bernoulli's random variables) treated in \cite{FIORR}. For this reason, we will trace a parallel between the latter  and the general case investigated here, pointing out and proving the main  modification needed in order to extend the proof to  any $s$ odd and general environment process $\boldsymbol{\xi}$ satisfying (UE).

The general idea is to bootstrap (R)$_M$, i.e., we show that the condition is satisfied for small values of $M$ and then use it to show that (R)$_{M+1/2}$ is satisfied as well. Once we have proven (R)$_M$, we combine it with \eqref{eq:L}, which says that $X$ is unlikely to decrease its distance from the root by a certain amount. In essence we show that the process is likely to behave as follows: if (R)$_M$ holds, in $n$ steps we are likely to see $X$ at distance $\log^M n$ away from the root; by \eqref{eq:L}  it is unlikely that in $n$ steps the walker backtrack half of this distance. Thus, instead of backtracking half the distance, the walker increases its distance by another $\log^M n$ and this  argument allows us to pass from (R)$_M$ to (R)$_{M+1/2}$.
\subsubsection{Small distance: Proving (R)$_{1/2}$} In the particular case of BGRW, at each step the walker has probability at least $p/2$ of increasing its distance by one: if it is on a leaf, it adds a new leaf with probability $p$ (since $s=1$ it has a chance of adding a new leaf at each step) and then jumps to it with probability $1/2$, and this is the worst scenario. Thus, if $M$ is small, in $n$ step we are likely to see the walker taking $\log^M n$ steps down in a row. For general $s$ odd we do not have this feature, so we overcome this by looking the process only at times multiple of $s$. The following lemma formalizes this argument.

\begin{lemma}\label{lemma:Mhalf} Consider a $\Name{}$ with $s$ odd and independent environment process satisfying condition \eqref{cond:UE}. Then, it satisfies $(R)_{1/2}$.
\end{lemma}

The proof of the above lemma relies on the following technical result.
\begin{lemma}[Lemma 3.6 of \cite{FIORR}]\label{lem:stochdom} Suppose $(I_j)_{j\in\N\backslash\{0\}}$ are indicator random variables. Assume $\mu$ is such that $\Pd\left(I_1=1\right)\geq \mu$ and
	\[\forall j>1\,:\,\Pd\left({I_j=1\mid I_1,\dots,I_{j-1}}\right)\geq \mu\;.\]
	Then for any $k,m\in\N\backslash\{0\}$ 
	\[\Pd\left(\mbox{at least $k$ consecutive $1$'s in the sequence $(I_j)_{j=1}^m$}\right)\geq 1 - (1-\mu^k)^{\lfloor m/k\rfloor}\;.
	\]
\end{lemma}

\begin{proof}[Proof of Lemma~\ref{lemma:Mhalf}] Define, for each  $1\leq k\leq \lfloor n/s \rfloor$, 
	\[
	I_{k}:=\mathbb{1}\{\dist{sk}{X_{sk}}{y} \geq \dist{s(k-1)}{X_{s(k-1)}}{y}+1\}\;,
	\]
and observe that by the Markov property:
\begin{equation*}
\Pd_{T_0,x_0,s, \bxi}\left(I_{k}=1\mid I_1,\dots,I_{k-1}\right)\geq \inf_{(T,x)\in \Omega}\Pd_{T,x,s, \bxi}\left(I_1=1\right)\geq \kappa \frac{1}{2}^{\left\lfloor \frac{s+1}{2}\right\rfloor}\geq \kappa 2^{-s}\;,
\end{equation*}
where the second inequality is justified by the following observation:  whenever $X$ is not on a leaf, it has probability at least $1/2$ of jumping down. Thus, for our bound we may consider the worst case possible which is $x_0$ is a leaf. Since our process satisfies condition  \eqref{cond:UE}, with probability at least $\kappa$ we add at least one leaf to $x_0$. Then, with probability at least $1/2$ we jump to one of the new leafs. Repeating this bouncing back argument on the leafs, paying at least $1/2$ to jump to a leaf and letting them push the walker back, we have that after $s$ steps $X_s$ is on a leaf of $x_0$ with probability at least $\kappa 2^{-\left\lfloor (s+1)/2\right\rfloor}$. 
Setting $k = \log^{\frac{1}{2}} n$, $m = n$ and $\mu = \kappa 2^{-s}$ in Lemma \ref{lem:stochdom}, proves the result. 
\end{proof}
\subsubsection{Small growth in distance} Now, we will show the key step to derive (R)$_{M+1/2}$ from~(R)$_{M}$. This relation relies on the following crucial lemma.
\begin{lemma}[Small growth in distance]\label{lem:key} 
	Consider a $\Name{}$ with  $s$ odd  and independent environment process satisfying condition \eqref{cond:UE}.  Also assume (R)$_M$ is satisfied for some $M\ge 1/2$. Then there exists $n_1(s, \bxi,M)\in\N$ such that, for $n\geq n_1$, the following property holds: 
	\[\Pd_{T_0,x_0, s, \bxi}\left(\exists t\leq n: \dist{t}{X_{t}}{y} = \dist{0}{x_0}{y}+1\right)\geq 1  - 2 \frac{(\log\log n)^2}{\log^M n}\;,\]
	for all finite tree $T_0$ and $x_0,y\in T_0$ with $\dist{0}{x_0}{y}\geq \log^M n$.
\end{lemma}
The above lemma basically says that, when (R)$_M$ is satisfied and $x_0$ is ``far'' from $y$, then it is likely that the distance between the walker and $y$ will increase by at least one unit by time $n$. This probability is large enough that we are likely to see many such increases in a small time window. 

The proof of Lemma \ref{lem:key} for the particular case BGRW is done in \cite{FIORR} and relies on condition (R)$_{M}$, which may be see as a \textit{global} condition since it gives information on how the walker $X$ is exploring/building the tree, and also on a \textit{local} feature of the \Name{} in this particular case: at each step $X$ has probability at least $p/2$ (where $p$ is the parameter of the model) of increasing its distance from the root by one unit.

This local feature is important because if the walker hits the bottom of the tree  many times then it is likely that after one of these hits it adds a new leaf and jump to it. However, in the case $s > 1$ this local feature is lost since the walk may hit the bottom with the wrong parity (at times not multiple of $s$) and then it goes back with probability one. Fortunately, a local correction is possible at the cost of a fixed probability depending on $s$. This is the core  of our next result and will be a key step for proving  Lemma \ref{lem:key}.
\begin{lemma}[Correcting the parity]\label{lemma:parity}Consider a $\Name{}$ with  $s$ odd  and independent environment process satisfying condition \eqref{cond:UE}, then
	\begin{equation*}
	\forall t \in \N, \inf_{T,x_0} \Pd_{T,x_0, s, \bxi}\left( \exists m \le 2s, \dist{t+m}{X_{t+m}}{y} = \dist{}{x_0}{y} + 1 \; \middle | \; X_{t} = x_0\right) \ge \frac{\kappa}{2^{s+1}}\;.
	\end{equation*}
\end{lemma}
Observe that for $s=1$ the lemma follows immediately, since the walker has probability at least $\kappa$ of attaching a new leaf on $x_0$ and probability at least $1/2$ of jumping to it. 
\begin{proof}[Proof of Lemma~\ref{lemma:parity}] We split the proof into cases.\\ 
	
	\noindent \underline{Case ${\rm deg}_{T}(x_0)\geq  2$.} We know that every time the walker visits $x_0$ it has probability
	at least $\frac{d-1}{d}\geq \frac{1}{2}\geq \frac{\kappa}{2^{s+1}}$ to jump to  a neighbor $x'$ of $x_0$ with 
	${\rm dist}(x',y)={\rm dist}(x_0,y)+1$. \\
	
	\noindent \underline{Case $t=0\, {\rm mod} \, s$.} The walker has probability at least $\kappa/2$ to attach at least one new leaf to $x_0$ and to jump to one of these new neighbor $x'$ of $x_0$ with ${\rm dist}(x',y)={\rm dist}(x_0,y)+1$.\\

	\noindent \underline{Case ${\rm deg}_{T}(x_0)=1  $ and $t=ls + r$, with  $0<r<s$.}
	
	We know that $x_0$ has a unique neighbor $x_1$ (belonging to the path connecting $x_0$ and $y$) and it must be the case that ${\rm deg}_{T}(x_1)\geq 2$. 
	We then consider two sub-cases: $s-r$ is {\em even} and $s-r$ is {\em odd}.  
	\begin{itemize}
		\item \underline{$s-r$ is {\em even}:} After visiting $x_0$ the walker necessarily will visit $x_1$. Then, with probability at least 
		$1/2$ the walker will visits one of the neighbors of $x_1$ (recall that ${\rm deg}_{T}(x_1)\geq 2$) which are not in the path connecting $x_1$ to $y$ ($x_0$ is also possible, and all of them have the same distance from $y$ as $x_0$).  It should be clear by now that the worst situation is when all such a neighbors are leaves (if not we have probability at least $1/2$ to increase further the distance from $y$, similarly to the case ${\rm deg}_{T}(x_0)\geq 2$)
		and therefore we are going to consider only this case. With probability  $1/2^{(s-r)/2}$, we have that $\dist{(l+1)s}{X_{(l+1)s}}{y} = \dist{}{x_0}{y}$.  Thus, with probability at least $\kappa$ the walker attaches a leaf on the vertex it resides on at time $(l+1)s$  and with probability $1/2$ it jumps to the new leaf. This proves that
		\begin{align*}
		\Pd_{T,x_0, s, \bxi}&\left( \dist{(l+1)s +1}{X_{(l+1)s+1}}{y} = \dist{}{x_0}{y} + 1 \;\middle | \; X_{ls +r } = x_0, {\rm deg}_{T}(x_0)=1 \right)
		\\
		& \ge \frac{\kappa}{2^{(s-r)/2+1}}\;.
		\end{align*}
		
		\item \underline{$s-r$ is {\em odd}:} this case is similar to the previous with the only difference that at time $(l+1)s$  the walker cannot resides on vertices with the same distance than $y$ as $x_0$, and it is necessary to take some extra steps. Note that
		with probability at least $1/2^{(s-r-1)/2}$ we have that $X_{(l+1)s} = x_1$. Then, taking other $s$ steps,  regardless the value of  $\xi_{(l+1)s}$, we still have probability at least $1/2^{(s+1)/2}$ of landing on $x_0$ or on one of the other leaves attached to $x_1$. This proves that
		\begin{equation*}
		\Pd_{T,x_0, s, \bxi}\left( X_{(l+2)s} \text{ is a leaf of }x_1 \;\middle | \; X_{ls +r } = x_0, {\rm deg}_T(x_0)=1 \right) \ge \frac{1}{2^{s}}\;.
		\end{equation*}
		Finally, with probability at least $\kappa$ we add leafs to $X_{(l+2)s}$ and with probability $1/2$ we jump to it. Then,
		\begin{align*}
		\Pd_{T,x_0, s, \bxi}&\left( \dist{(m+2)s+1}{X_{(l+2)s+1}}{y} = \dist{}{x_0}{y} + 1 \;\middle | \; X_{ls +r } = x_0, {\rm deg}_T(x_0)=1  \right) 
		\\ & \ge \frac{\kappa}{2^{s+1}}\;.
		\end{align*} 
	\end{itemize}
\end{proof}
Now we are able to prove Lemma~\ref{lem:key}.
\begin{proof}[Proof of Lemma \ref{lem:key}] 

	Denote by $y_*$ the vertex on the unique path from $x_0$ to $y$ with~$\dist{}{x_0}{y_*}=\lceil \log^M n\rceil -1$. By condition (R)$_M$, there exists $n_0$ depending only on $s,M$ and $\bxi$ such that:
	\begin{equation*}
	\forall n\geq n_0\,:\,\Pd_{T,x_0,s,\bxi}(\exists t\leq n\,:\, \dist{t}{X_{t}}{y_*} \geq \log^M n) \geq 1 - e^{-n^{1/4}}\;.
	\end{equation*}
	Let $F$ denote the event that the walker has failed to increase its distance from $y$, i.e, the event that $\dist{t}{X_{t}}{y} \leq \dist{}{x_0}{y}$ for all $t\leq n$. Let~$\tau_{y_*}$ be the hitting time of $y_*$ 
	\[\tau_{y_*}:=\inf\{t\in\N\,:\, X_{t}=y_*\}\in \N\cup\{+\infty\}\;.\]

	We also define inductively $\tau^{+k}_{x_0}$ as the $k$-th return time to $x_0$. Setting $\tau^{+0}_{x_0}:=0$, $\tau^{+k}_{x_0}$ is defined as
	\[\tau^{+k}_{x_0}:=\left\{\begin{array}{ll}+\infty, & \mbox{ if }\tau_{x_0}^{+(k-1)}=+\infty\;,\\ \inf\{t> \tau^{+(k-1)}_{x_0}\,:\,X_{t}=x_0\}\in \N\cup\{+\infty\}, & \mbox{otherwise}\;.\end{array}\right.\]
	Observe that, for each $k$, we bound the probability of $F$ as follows
	\begin{align}\label{eq:failureclaim3}
	\Pd_{T,x_0,s, \bxi}(F)&\leq\Pd_{T,x,s, \bxi}(F\cap \{\tau_{y_*}>n\}) + \Pd_{T,x_0,s, \bxi}(F\cap \{\tau^{+k}_{x_0}<\tau_{y_*}\leq n\})\\
	&\nonumber + \Pd_{T,x_0,s, \bxi}( F \cap \{\tau^{+k}_{x_0}\geq\tau_{y_*}\})\;.
	\end{align}
	The bounds for the first and third terms in the RHS are the easiest ones. For the first one, observe that in that event the walker has not achieved distance at least $\log^Mn$ from~$y_*$ in $n$ steps. Then, by condition (R)$_M$, this happens with probability at most $\exp \{-n^{1/4}\}$. Whereas, for the third one, note that in this event, before the $k$-th visit of $x_0$ the walk $X$ has reached $y_*$. By a simple comparison with a simple random walk on the line connecting~$x_0$ to~$y_*$ we bound the third term by $k/ (\lceil \log^Mn\rceil-1)$.

	We finally  consider the second term in the RHS of (\ref{eq:failureclaim3}). In order for $F\cap \{\tau^{+k}_{x_0}<\tau_{y_*}\}$ to take place, it must be that $X_t$ returns at least $k$ times to $x_0$ before visiting $y_*$ {\em but} never gets to jump to a neighbor of $x_0$ which does not belong to the unique path connecting $x_0$ to $y$. Note that if $s=1$ and (UE) holds, then at each visit to $x_0$ we have a bounded away from zero probability of jumping down. Thus, in this particular case, the second term of \eqref{eq:failureclaim3} decays exponentially fast in $k$. To extend this idea to for general $s$ odd we apply Lemma~\ref{lemma:parity} in the following way:
	
	Let $A_t$ denote the following event
	\begin{equation*}
	A_t := \left \lbrace \exists m \le 2s, \dist{t+m}{X_{t+m}}{y} = \dist{}{x_0}{y} + 1 \right \rbrace.
	\end{equation*}
	By Lemma~\ref{lemma:parity}, we have that, for all $t \in \N$
	\begin{equation}\label{eq:At}
	\inf_{T,x_0} \Pd_{T,x_0, s, \bxi}\left( A_t \; \middle | \; X_t = x_0 \right) \ge \frac{\kappa}{2^{s+1}}\;.
	\end{equation}
	Also notice that the following inclusion of events holds
	\begin{equation*}
	F \cap \{\tau^{+k}_{x_0} < \tau_{y_*} \le n\} \subset \bigcap_{j=0}^{\lfloor k/2s\rfloor} \left(A^c_{\tau^{+2sj}_{x_0}}\cap \{\tau^{+2sj}_{x_0}< n\} \right)\;.
	\end{equation*}
	Combining the Strong Markov Property with \eqref{eq:At} leads to
	\begin{equation*}
	\begin{split}
	\Pd_{T,x_0, s, \bxi} \left( A^c_{\tau^{+k}_{x_0}} \; \middle | \; \tau^{+k}_{x_0} < n, \bigcap_{j=0}^{\lfloor k/2s \rfloor-1} \left(A^c_{\tau^{+2sj}_{x_0}}\cap \{\tau^{+2sj}_{x_0}< n\} \right)\right) \le 1 - \frac{\kappa}{2^{s+1}}\;.
	\end{split}
	\end{equation*}
	The above bound implies that
	\begin{equation*}
	\Pd_{T,x_0, s, \bxi} \left(F \cap \{\tau^{+k}_{x_0} < \tau_{y_*} \le n\}\right) \le \left(1 - \frac{\kappa}{2^{s+1}}\right)^{\lfloor k/2s \rfloor} \le e^{\log(1-\kappa/2^{s+1})\lfloor k/2s \rfloor}\;.
	\end{equation*}
	Overall, we obtain
	\[
	\Pd_{T,x_0,s, \bxi}(F) \leq e^{-n^{1/4}} + \frac{k}{\lceil \log^Mn\rceil-1} + e^{\log(1-\kappa/2^{s+1})\lfloor k/2s\rfloor}\;.
	\]
	Setting $k = \log\log^2n$, proves the lemma.
\end{proof}
\subsubsection{Iterating the argument. 
} 
In Lemma~\ref{lemma:Mhalf}, we have proved that under (UE), and $s$ odd, (R)$_{1/2}$ holds. Now, in order to prove that \Name{} satisfies condition (R)$_M$, for some $M>1$ (and thus condition \eqref{eq:R}) we prove  the following lemma.

\begin{lemma}\label{lemma:iterating} Consider a $\Name{}$ with $s$ odd and independent environment process satisfying condition \eqref{cond:UE}. Then, if (R)$_{M}$ holds, so does (R)$_{M+1/2}$.
\end{lemma}

\begin{proof}[Proof of Lemma~\ref{lemma:iterating}] Once we have Lemma~\ref{lem:key}, the proof that (R)$_M$ implies (R)$_{M+1/2}$ is similar with the proof for BGRW provided in \cite{FIORR}. We give here just a sketch, and refer the reader to Proposition 3.4 in \cite{FIORR} for further details.
	
\medskip 
	
Observe that Lemma~\ref{lem:key}, replacing $n$ by $\sqrt{n}$, yields
\begin{equation*}
\Pd_{T_0,x_0, s, \bxi}\left(\exists t\leq \sqrt{n}: \dist{t}{X_{t}}{y} = \dist{0}{x_0}{y}+1\right)\geq 1  - 2^{M-1} \frac{(\log \log n )^2}{\log^M n}\;,
\end{equation*}
whenever $x_0$ is at distance at least $2^{-M}\log^M n$ from $y$. 
To guarantee we can actually use the above bound, we may use condition  (R)$_M$ for~$\sqrt{n}$ to obtain that distance $2^{-M}\log^M n$ from any vertex of $T_0$ is likely to be achieve in at most $\sqrt{n}$ steps,  i.e., we have that 
\begin{equation*}
\inf_{(T_0,x_0) \in \Omega} \mathbb{P}_{T_0, x_0, s, \bxi}\left(\exists m \leq \sqrt{n} : \dist{m}{X_{m}}{y}\geq 2^{-M}\log^M n  \right)\geq 1-e^{-n^{1/8}}\;,
\end{equation*}
for any $y \in T_0$. Thus, for $n$ sufficiently large, we may assume that the walker is at distance at least $2^{-M}\log^M n$ from $y \in T_0$. Once this is the case, the core of the argument is that the probability of increasing  the distance by one unit in $\sqrt{n}$ steps is high enough so that we are likely to see many consecutive such increase, and specifically, it is likely to see the walker increasing by one its distance $\log^{M+1/2} n$ times in a roll. This intuition may be  formalized using  Lemma~\ref{lem:stochdom}, letting the $I$'s to be the indicator of the event the walker has increased its distance from  $y$ by one unit in less than $\sqrt{n}$ steps, setting~$\mu = 1  - 2^{M-1} \frac{(\log\log n)^2}{\log^M n}$, $k=\log^{M+1/2}n$ and $m=\sqrt{n}$ (since in a time window of range $n$ we have at least $\sqrt{n}$ trials), we obtain essentially that
\begin{equation*}
\begin{split}
\Pd_{T_0,x_0,s,\bxi}\left(\mbox{at least $k$ consecutive $1$'s in the sequence $(I_j)_{j=1}^m$}\right) & \geq 1 - (1-\mu^k)^{\lfloor m/k\rfloor}\;.
\end{split}
\end{equation*}
Note that $k$ consecutive $1$'s implies that at some moment the walker is at distance at least~$k=\log^{M+1/2}n$ from  $y$. Moreover, our choices for $\mu$ and $m$ leads to
\[
(1-\mu^k)^{\lfloor m/k\rfloor}
\leq \exp\left(- \mu^{k} \frac{n^{1/2}}{2(\log^{M+1/2} n + 1)}\right)\;. 
\]
Since $k=\lceil \log^{M+1/2} n\rceil$, for sufficiently large $n$ we have
\[ \mu^{k} \geq \exp\left( -\left( 2(\log\log n)^2 2^{M+1} + o((\log\log n)^2)\right) \sqrt{\log n}\right)\mu  
=n^{-o(1)}\mu\;,
\]
using that $\left(1-\frac{b_n}{a_n}\right)^{a_n}\approx e^{-b_n -o(b_n)}$ for sufficiently large $n$ whenever $a_n, b_n \to \infty$ and $b_n=o(a_n)$. Overall, we obtain that
\[
(1-\mu^{k})^{\left\lfloor \frac{\lfloor \sqrt{n}\rfloor}{k}\right\rfloor}\leq \exp\left(-n^{1/4}\, \frac{n^{1/4-o(1)}\mu}{2(\log^{M+1/2} n+1)}\right) 
\leq \exp\left(-n^{1/4}\right)\;,
\]
for large enough $n$.
\end{proof}

%% file: height.tex

\section{Structural knowledge: the environment growth}\label{sec:height}
In this section we analyze the growth of the sequence of rooted random trees $\{T_n\}_{n \in \N}$ generated by a \Name{} process. We denote by $h$ the \textit{height functional} defined for each tree $T$ as $h(T) = \max_x \dist{}{x}{\Root}$. From Theorem~\ref{theo:ballistic} if $\s$ is odd and the  environment process  satisfies condition (\ref{cond:UE}) then we have that
\[
\liminf_{n \to \infty}\frac{h(T_n)}{n} > 0\;, \; \Pd_{T_0, x_0, s, \boldsymbol{\xi}}-a.s.\;.
\]
This means that the height of the generated trees grows linearly in time. On the other hand, if $s$ is even and the environment satisfies condition \eqref{cond:I},  by Theorem~\ref{rec-even} item $(ii)$, the tree height stops growing almost surely. What does happen to the sequence of random variables~$\{h(T_n)\}_{n \in \N}$ under condition \eqref{cond:S}?   To begin answering the latter question, we recall the example from Section \ref{sec:results}: the  independent environment $\bxi$ with $\xi_j \sim {\rm Ber}(j^{-2})$  satisfies condition \eqref{cond:S}, however the sequence $\{h(T_n)\}_{n \in \N}$ is almost surely finite for any value of $\s$, since the process, eventually, stops adding new leaves. To avoid the above situation, we impose the additional condition \eqref{cond:UE} on the environment and prove the following proposition.

\begin{proposition}\label{prop:height} Consider a {$(2k,\boldsymbol{\xi})$-\Name{}} process whose environment process $\boldsymbol{\xi}$ satisfies conditions \eqref{cond:S} and \eqref{cond:UE}. Then, for every initial state  $(T_0,x_0)$
	\begin{equation*}
	h(T_n) \nearrow + \infty\;,
	\end{equation*}
$\Pd_{T_0, x_0, s, \boldsymbol{\xi}}$-almost surely.
\end{proposition}
\begin{proof}
Observe that a vertex which maximizes the height of the tree is necessarily a leaf,
whose parent has only leaves as children. Call this parent $z$. Then $z$ will be visited infinitely many times, but we need to guarantee that it will be visited infinitely many times with opposite parity as that of the walker. Indeed, when $z$ is visited with opposite parity, then with probability greater than $1/2$ one of its leaves will be visited on the next transition when the walker will have the same parity as this visited leaf. Therefore, at any of these times when the walker visits $z$ with opposite parity, we have a probability of at least  $2^{-\s/2}$ to jump to one of its leaves at times multiple of $s$, and then  a probability of at least $\kappa$, granted by condition \eqref{cond:UE}, to increase by one the height of the tree. If this last event occurs with probability one, then also with probability one the height will increase indefinitely.

Under condition \eqref{cond:S} on the environment, consider a realization of the \Name{} from time $0$ to time $k$, for some fixed $k \ge 0$. Let $z \in T_k$ such as before. We have to show that $z$ is visited infinitely many times with opposite parity as that of the walker. From now on we call this parity the ``right parity''. Put $y$ as the neighbor of the $\Root$ such that $z$ belongs to the branch of $T_k$ starting at $y$. Following our proof of recurrence, we only have to show that $y$ is visited infinitely many times when the walker has the right parity. By Lemma \ref{lem:neighbors-recurrence} and Corollary~\ref{cor:recur} both the edges $e_1 = (\Root,\Root)$ and $e_2 = (\Root,y)$ are traversed infinite many times. Define the sequence of Bernoulli random variables $(a_j)_{j\ge 1}$ as $a_j = 1$ if on the $j$-th crossing, after time $k$, of either $e_1$ or $e_2$, the edge crossed is $e_1$, otherwise set $a_j=0$. The sequence $(a_j)_{j\ge 1}$ is i.i.d. From its mixing property, $y$ is visited on both even and odd times infinitely often, which implies that $y$ is visited infinitely many times when the walker has the right parity.
\end{proof}

\medskip

{We finish this section proving Proposition \ref{prop:null-recurr} which states that, for $s$ even, condition~(\ref{cond:S}) combined with (\ref{cond:UE}) implies null recurrence.
\begin{proof}[Proof of Proposition \ref{prop:null-recurr}] 
If the environment satisfies condition \eqref{cond:UE}  we can use  Lemma~\ref{lem:infinite-expectation}, assuring that the expected hitting time of vertices sufficiently far away from the initial state are infinite. By Proposition~\ref{prop:height}, under conditions (\ref{cond:S}) and (\ref{cond:UE}) the height of the tree goes to infinity. Therefore, eventually the walker will be sufficiently far away from any vertex.  Using the strong Markov property, the proof is complete.
\end{proof}
}

%% file: finalcomments.tex

\section{Final comments} \label{sec:fincom}
We end this paper making some comments regarding the finiteness of the initial tree~$T_0$, the $\Name{}$ seen as a random graph process as well as the elliptic case in the context of $\Name{}$. We believe these comments could lead to interesting questions. 
\subsection{Finiteness of $T_0$}
Recall that in the results for $s$ even from Section~\ref{sec:rec}, we required the initial tree~$T_0$ to be finite. However the definition in Section \ref{sec:model} consider any initial locally finite tree. More generally, the initial state of the $\Name$ may be sampled according to some distribution $\nu$ over the space of pairs $(T, x)$, where $T$ is a locally finite tree, possibly infinite.
	However, allowing infinite trees may lead to different questions from those we have addressed in this paper. For instance, in the infinite case one may not observe the trapped regime we proved for some environment conditions when $s$ is even, as is illustrated in the example below.
	\begin{example}[A heavy infinity tree] For each natural $n\ge 1$, let $d_n$ be $n^4$. Now, consider the infinity tree $T_{\infty}$ defined recursively in such way that all vertices at level $n$ have degree~$d_n$. Also consider the deterministic environment $\boldsymbol{\xi}$ defined by $\xi_{sj} = j^2$. Thus, $S_n \approx n^3$ and consequently the $(\boldsymbol{\xi},s)-\Name$ model satisfies condition \eqref{cond:I} for $f(n) = n^3$. However, regardless the parity of $s$, we have the following 
		\begin{equation*}
		\Pd_{T_{\infty}, {\rm root}, s, \boldsymbol{\xi}}\left( \dist{n}{X_n}{{\rm root}} = n \right) \ge  \prod_{j=1}^{n}\left(1-\frac{1+j^2}{j^4}\right) > 0\;.
		\end{equation*}
		Thus, the walker has positive probability of ignoring all the leaves it adds to $T_{\infty}$ and simply goes ``down the tree".
	\end{example}
In the above example probability of always going down can be made arbitrarily close to one by considering even heavier trees. Thus, for a $\Name$ process whose initial state is sampled from any distribution $\nu$ supported on heavy infinity trees, Theorem~\ref{rec-even} does not hold. This discussion leads naturally to the question: What are the conditions over the distribution $\nu$ in order that Theorem~\ref{rec-even} still holds?

Of course our results cover the case in which $\nu$ is supported on the subset of finite trees. The significant contribution here would be the case that involves not only ``well-behaved" infinite trees.

\subsection{Random tree process}Except from Section~\ref{sec:height}, all the results in this paper regard the walker $X$, and we have approached the $\Name{}$ process as a process of random walk on random environment. {However, one may also study the $\Name$   from the perspective of the random trees $\{T_n\}_{n \in \N}$. When focusing on the random trees $\{T_n\}_{n \in \N}$,  $\Name$ can be seen  as a random graph model: starting from $T_0$, at every time $n$ multiple of $s$, a random number $\xi_n$ of new leaves  are added to the current tree and attached to the vertex where the random walk $X$ resides at that time (see, Section~\ref{sec:model}).  When studying $\Name$ from the perspective of the random trees $\{T_n\}_{n \in \N}$   questions concerning the \textit{structure} and \textit{degree distribution} of the random sequence of trees $\{T_n\}_{n\geq 0}$ stand out.
 For instance, it would be interesting to study if  $\Name{}$ is capable of generating trees whose degree distribution obeys a power law. That is, under which assumption on the environment $\{\xi_n\}_{n\geq 0}$, there exist a positive constant $C$ and positive exponent $\gamma$ (possibly depending on $\{\xi_n\}_{n\geq 0}$) such that for every fixed $d\geq 1$ 
 \begin{equation*}
	\lim_{n \rightarrow \infty} \frac{1}{|V(T_n)|}\sum_{x\in V(T_n)}\mathbb{1}\{\degr{n}{x} = d\}  = C d^{-\gamma} \;, \quad \Pd_{T_0, x_0, s, \boldsymbol{\xi}}\text{- a.s.}\;,
	\end{equation*}
	where, we recall that, $V(T_{n})$ denotes the vertex set of $T_n$  and for $x \in V(T_n)$, $\degr{n}{x}$ denotes the degree of $x$ in $T_n$.
 }

\subsection{Ellipticity}In the context of the classical RWRE, efforts have been made towards dropping the {\em uniformly} elliptic condition. For instance, in \cite{BRS16,CR14, FK2016} authors have obtained ballisticity criteria under elliptic condition.

On the other hand, in the context of $\Name{}$, an ellipticity condition means that the probability of adding at least one leaf is positive for each time multiple of $s$ but it vanishes in the long run. More formally, we could define the following environment condition
\begin{equation}\tag{E}\label{cond:E}
\lim_{n\to \infty} P \left( \xi_{n}\geq 1\right)=0\;.
\end{equation}
It is clear that if $P\left( \xi_{n}\geq 1\right)$ goes fast enough to zero, Borel-Cantelli lemma implies that the process is positive recurrent since the walker stops to add new leaves to the graph eventually. The interesting question here would be to find other regimes for the decreasing rate of $P\left( \xi_{n}\geq 1\right)$ for which zero speed and ballisticity are also observed.

%% file: main.bbl
\begin{thebibliography}{99}


\bibitem{A2018}  L. Avena, H. G{\"u}lda{\c{s}}, R. van der Hofstad, F. den Hollander: \textit{Mixing times of random walks on dynamic configuration models}, The Annals of Applied Probability, 28 (4), 1977-2002, (2018).

\bibitem{BK}  L. Baum, M. Katz: \textit{Convergence Rates in the Law of Large Numbers}, Transactions of the American Mathematical Society, 120, 108-123, (1965).

\bibitem{BRS16} {\'E} Bouchet, A. F. Ram{\'i}rez, and C. Sabot: \textit{Sharp ellipticity conditions for ballistic behavior of random walks in random environment}. Bernoulli, 22(2), 969-994,(2016).

\bibitem{CR14} D. Campos and A. F. Ram{\'i}rez: \textit{Ellipticity criteria for ballistic behavior of random walks in random environment}, Probability Theory and Related Fields, 160(1), 189-251, 2014.

\bibitem{D} B. Davis:\textit{ Reinforced Random Walk}, Probability Theory and Related Fields, 84(2), 203-229, (1990).

\bibitem{DHS} A. Dembo, R. Huang and V. Sidoravicius. \textit{Walking within growing domains: recurrence versus transience}. Electronic Journal of Probability, 19(106), 20 pp, (2014).

\bibitem{Du} R. Durrett: {\em Probability: Theory and Examples}, 4th ed., Cambridge University Press, (2010). 
\bibitem{Du2} R. Durrett: {\em Random Graphs Dynamics}, Cambridge University Press, (2007).
\bibitem{DST} M. Disertori, C. Sabot, P. Tarres: \textit{Transience of edge-reinforced random walk},  Communications in Mathematical Physics, 339(1), 121-148, (2015).
\bibitem{EL} P. Eichelsbacher, M. L\"owe: \textit{Moderate deviations for IID random variables}, ESAIM: Probability and Statistics,  7, 209-218, (2003).

\bibitem{FIN} D. Figueiredo, G. Iacobelli, G. Neglia: \textit{Transient and slim versus recurrent and fat: random walks and the trees they grow}, Journal of Applied Probability, 56(3), 769-786,  (2019).

\bibitem{FIORR} D. Figueiredo, G. Iacobelli, R. Oliveira, B. Reed, R. Ribeiro: \textit{On a random walk that grows its own tree}, Electronic Journal of Probability, 26(6), 40 pp, (2021). 

\bibitem{FK2016}A. Fribergh and D. Kious: \textit{Local trapping for elliptic random walks in random environments in $\mathbb {Z}^ d$}. Probability Theory and Related Fields, 165(3-4), 795-834, (2016).

\bibitem{RP1988} R. Pemantle: \textit{Phase transition in reinforced random walk and RWRE on trees}. The Annals of Probability, 16(3), 1229-1241,  (1988).

\bibitem{RP2007} R. Pemantle: \textit{A survey of random processes with reinforcement}. Probability surveys, 4, 1-79,  (2007).


\bibitem{H} R. van der Hofstad. \textit{Random graphs and complex networks (Vol. 1)}. Cambridge university press, (2016).


\bibitem{JLR} S. Janson, T. Luczak and A. Rucinski: \textit{Random graphs}. John Wiley \& Sons, 45, (2011).


\bibitem{KZ} E. Kosygina and M. Zerner: Excited random walks: results, methods, open problems. Bulletin of the Institute of Mathematics, Academia Sinica (New Series), 8(1), 105-157, (2013).

\bibitem{LPP} R. Lyons, R. Pemantle and Y. Peres: \textit{Biased random walks on Galton--Watson trees}. Probability Theory and Related Fields, 106 (2), 249-264, (1996).

\bibitem{PSS} Y. Peres, A. Stauffer and J.E. Steif. \textit{Random walks on dynamical percolation: mixing times, mean squared displacement and hitting times}. Probability Theory and Related Fields, 162 (3-4), 487-530, (2015).

\bibitem{S} A. S. Sznitman:\textit{ On a class of transient random walks in random environment}, The Annals of Probability, 29(2), 724-765, (2001).

\bibitem{SZ} A. S. Sznitman and M. Zerner. \textit{A law of large numbers for random walks in random environment}. The Annals of Probability, 1851-1869, (1999).

\bibitem{OZ} O. Zeitouni. \textit{Random walks in random environment}. Lecture notes in Mathematics, 1837, 190-312, (2004).


\end{thebibliography}
